\documentclass[12pt,preprint]{elsarticle}

\usepackage[T1]{fontenc}
\usepackage{graphicx}
\usepackage{latexsym}
\usepackage{amsmath}
\usepackage{amsxtra}
\usepackage{amssymb}
\usepackage{psfrag}
\usepackage{natbib}

\newcommand{\w}{\omega}

\newtheorem{define}{Definition}
\newtheorem{propos}[define]{Proposition}
\newtheorem{remrk}[define]{Remark}

\newtheorem{prob}[define]{Problem}
\newcommand{\C}{\mathbb{C}}
\newcommand{\Z}{\mathbb{Z}}
\newcommand{\N}{\mathbb{N}}
\newcommand{\R}{\mathbb{R}}

\graphicspath{{./figure/}}

\begin{document}

\journal{Systems and Control letters}
\begin{frontmatter}
\title {Model Identification of a network as compressing sensing}

\author[umn]{D.~Materassi}
\author[si]{G.~Innocenti}
\author[la]{L. Giarr\'e}
\author[umn]{M. Salapaka}

\address[si]{
  Dipartimento di Ingegneria dell'Informazione \\
  Universit\'a di Siena, \\
  via Roma 56, 53100 Siena, Italy\\
  E-mail: innocenti@dii.unisi.it
}
\address[umn]{
  Department of Electrical and Computer Engineering,\\
  University of Minnesota,\\
  200 Union St SE, 55455, Minneapolis (MN) \\
  E-mail: mater013@umn.edu
}

\address[la]{
  Dipartimento di Ingegneria Elettrica, Elettronica e delle Telecomunicazioni\\
  Universit\'a di Palermo,\\
  viale Delle Scienze, 90128 Palermo, Italy \\
  E-mail: giarre@unipa.it
}

\begin{abstract}
\noindent
In many applications, it is important to derive information about the topology and the internal connections of  dynamical systems interacting together.
	Examples can be found in fields as diverse as Economics, Neuroscience and Biochemistry.
	The paper deals with the problem of deriving a descriptive model of a network, collecting the node outputs as time series with no use of a priori insight on the topology, and unveiling an unknown structure as the estimate of a ``sparse Wiener filter''. 	A geometric interpretation of the problem in a pre-Hilbert space for wide-sense stochastic processes is provided. 	We cast the problem as the optimization of a cost function where a set of parameters are used to operate a trade-off between accuracy and complexity in the final model.
	The problem of reducing the complexity is addressed by fixing a certain degree of sparsity and finding the solution that ``better'' satisfies the constraints according to the criterion of approximation. Applications starting from real data and numerical simulations are provided.
\end{abstract}
\end{frontmatter}



\section{Introduction}\label{sec:intro}
The interest on networks of dynamical systems is increasing in recent years, especially because of their capability of modeling and describing a large variety of phenomena and behaviors.
Remarkably, while networks of dynamical systems are well studied and analyzed in physics \citep{BocLat06,GirNew02,NewGir04} and engineering \citep{ZhaLiu07,Olf07,SchRib08}, there are fewer results that address the problem of reconstructing an unknown dynamical network, since it poses formidable theoretical and practical challenges \citep{Kol09}.
However, unraveling the interconnectedness and the interdependency of a set of processes is of significant interest in many fields and the necessity for general tools is rapidly emerging (see \cite{Tim07,BocIva07,NapSau08} and the bibliography therein for recent results).
In the literature, authors have approached this problem in different ways and with various purposes, such as deriving a network topology from just sampled data (see e.g. \cite{ManSta00,Tim07,NapSau08,OzeUzu08}) or determining the presence of substructures in the networked system (see e.g. \cite{NewGir04,BocIva07}).
The Unweighted Pair Group Method with Arithmetic mean (UPGMA) \citep{MicSok57} is one of the first techniques proposed to reveal an unknown topology.
It has found widespread use  in the reconstruction of phylogenetic trees and is widely employed in other areas such as communication systems and resource allocation problems \citep{FreHar02}.
Another well-known technique for the identification of a tree network is developed in \cite{ManSta00} for the analysis of a stock portfolio.
The authors identify a tree structure according to the following procedure: i) a metric based on the correlation index is defined among the nodes; ii) such a metric is employed to extract the Minimum Spanning Tree \cite{Die06} which forms the reconstructed topology.
However, in \cite{InnMat08} a severe limit of this strategy is highlighted,
where it is shown that, even though the actual network is a tree, the presence of dynamical connections or delays can lead to the identification of a wrong topology.
In \cite{MatInn09} a similar strategy, where the correlation metric is replaced by a metric based on the coherence function, is numerically shown to provide an exact reconstruction for tree topologies. Furthermore, in \cite{MatInn08} it is also illustrated that a correct reconstruction can be guaranteed for any topology with no cycles.\\
An approach for the identification of more general topologies is developed in the area of Machine Learning for Bayesian dynamical networks \citep{GetFri02,FriKol03}.
In this case, however, a massive quantity of data needs to be collected in order to accurately evaluate conditional probability distributions.\\
In \cite{BocIva07} different techniques to quantify and evaluate the modular structure of a network are compared and a new one is proposed trying to combine both the topological and dynamic information of the complex system.
However, the network topology is only qualitatively estimated in terms of ``clusters'', \cite{blondel}.
In \cite{Tim07} a method to identify a network of dynamical systems is described. However, primary assumptions of the technique are the possibility to manipulate the input of every single node and to conduct as many experiments as needed to detect the link connectivity.\\
More recently, in \cite{NapSau08} and \cite{mat09} interesting equivalences between the identification of a dynamical network and a $l_0$ sparsification problem are highlighted, suggesting the difficulty of the reconstruction procedure \citep{CanTao05,CanWak08}.

In this paper, the main idea is to cast the problem of unveiling an unknown structure as the estimate of a ``sparse Wiener filter''.
Given a set of $N$ stochastic processes $\mathcal{X}=\{x_1,...,x_n\}$, we consider each $x_j$ as the output of an unknown dynamical system, the input of which is given by at most $m_j$ stochastic processes $\{x_{\alpha_{j,1}}, ..., x_{\alpha_{j,m_j}}\}$ selected from $\mathcal{X}\setminus \{x_j\}$.
The choice of $\{x_{\alpha_{j,1}}, ..., x_{\alpha_{j, m_j}}\}$ is realized according to a criterion that takes into account the mean square of a modeling error. The parameters $m_j$ can be a-priori defined, if we intend to impose a certain degree of sparsity on the network or a strategy for self-tuning can be introduced penalizing the introduction of any additional link, if it does not provide a significant reduction of the cost.
For any possible choice of $\{x_{\alpha_{j,1}}, ..., x_{\alpha_{j, m_j}}\}$, the computation of the related Wiener Filter leads to the definition of a modeling error, which is a natural way to measure the quality of the description of $x_j$ granted by the time series $\{x_{\alpha_{j,1}}, ..., x_{\alpha_{j, m_j}}\}$ in terms of predictive/smoothing capability.
Once this step has been performed, each system is represented by a node of a graph and, then, the arcs linking any $x_{\alpha_{j, m_k}}$ to $x_j$ are introduced for each node $x_j$.
At the end of this procedure a graph, modeling the network topology, has been obtained.\\
We start introducing a pre-Hilbert space for wide-sense stochastic processes, where the inner product defines the notion of perpendicularity between two stochastic processes.
We will show that this way of formulating the problem has strong similarities with $l_0$-minimization problems, which have been a very active topic of research in Signal Processing during the last few years.
Indeed, a standard $l_0$-minimization problem amounts to finding the ``sparsest'' solution of a set of linear equations in a finite dimension Hilbert space \cite{CanWak08}.
With no additional assumptions on the solution, the problem is combinatorially intractable \cite{CanWak08}. This has propelled the study of relaxed problems involving, for example, the minimization of the $\ell_1$ norm, which is a convex problem and it is known to provide solutions with at least a certain order of sparsity \cite{NapSau08}. Unfortunately, we will show that
such a relaxation procedure is not viable in our formulation. Indeed, it is not possible to define a suitable norm in the space of transfer functions that guarantees a certain degree of sparsity. For this reason, we resort to some greedy techniques in order to find a suboptimal solution with desired sparsity properties.




The rest of the paper is organized as follows. 

%

In Section \ref{sec:problem} the network topology identification problem is formulated.
In Section \ref{sec:geometric} a geometric interpretation and the construction of a pre-Hilbert space needed 
 to define a distance and an inner product for stochastic processes is addressed.
 In Section \ref{sec:links} the connection with the compressive sensing problem is shown. 
In Section \ref{sec:solution via suboptimal} a greedy algorithm addressing the problem is presented along with an alternative approach based on iterated re-weighted least squares. Finally, in Section \ref{sec:examples} the results obtained by applying the techniques to numerical data are discussed. In the Appendix most of the needed definitions, propositions,  lemmas and proofs needed for the construction of the pre-Hilbert space are added. \\
\noindent{\bf Notation:}\\
$\Z$: the integer set;\\
$\R$: the real set;\\
$\C$: the complex set;\\
$E[\,\cdot\,]$: the mean operator; \\
$(\,\cdot\,)^T$: the transponse operator.
\section{Problem Formulation}\label{sec:problem}

In this section we provide the main definitions to cast the problem of modeling a network structure. We consider $M$ stochastic processes $x_1, \ldots, x_M$ representing the output of $M$ nodes in a network with an unknown topology. In order to determine the links connecting the nodes, we follow a procedure based on estimation techniques. Given a process $x_j$ and a parameter $m_j \in N$, we search for the $M_j$ processes $x_{\alpha_1},\ldots, x_{\alpha_k}, \ldots, x_{\alpha_{m_j}}$ with $\alpha_k \neq j, \forall k=1,\ldots,m_j$, which provide the best estimate of $x_j$ according to a quadratic criterion. The value $m_j$ is a tuning parameter allowing one to operate a trade-off between the sparsity and the accuracy of the model. Thus, $m_j$ can be a-priori chosen or, conversely, determined using a self-tuning strategy. 

We introduce now some definitions and results, which turn out essential for the rigorous formulation of the  problem.
For sake of clarity, we report in the Appendix all the additional needed definitions and properties.

\begin{define}
	Let $e_{i}(t)$, with $i=1,...,N$ and $t\in\Z$, be $N$ scalar time-discrete, zero-mean, jointly  wide-sense stationary random processes in a probability space $(\Omega, \sigma, \Pi)$, where $\Omega$ is the sample space, $\sigma$ is a sigma algebra on $\Omega$ and $\Pi$ is a probability measure on $\sigma$.
	Then, for any $t\in\Z$ define the vector $e(t):=(e_1(t),..,e_{N}(t))^T$, describing a $N$-dimensional time-discrete, zero-mean, wide-sense stationary random process. Moreover, for any $t_1, t_2 \in \Z$ denote the $(N \times N)$ covariance matrix as
	\begin{align}\label{eq:covariance e(t)}
		R_{e}(t_1, t_2):=E[e(t_1)e^T(t_2)].
	\end{align}
	The entry $(i,j)$, with $i,j\in \{1,...,N \}$ of $R_{e}(t_1,t_2)$ is given by
	\begin{align*}
		R_{e_i e_j}(t_1,t_2):=E[e_i(t_1) e_j(t_2)].
	\end{align*}
	Since any two processes $e_i$ and $e_j$ are jointly wide-sense stationary by definition, $R_{e_i e_j}(t_1, t_2)$ only depends on $\tau:=t_2-t_1$:
	\begin{align*}
		R_{e_i e_j}(0,t_2-t_1)=R_{e_i e_j}(t_1, t_2), \qquad \forall t_1,t_2 \in \Z,
	\end{align*}
	and, thus, $R_{e}(t_1, t_2)$ too depends only on $t_2-t_1$, i.e.
	\begin{align*}
		R_{e}(0,t_2-t_1)=R_{e}(t_1, t_2) , \qquad \forall t_1,t_2 \in \Z \ .
	\end{align*}
	Abusing the notation it is possible to write more concisely $R_{e}(\tau)=R_{e}(0,\tau)$.
\end{define}

\begin{define}
	Consider a vector-valued sequence $h(k) \in \R^{1\times N}$ with $k\in \Z$. We define its $\mathcal{Z}$-transform as:
	\begin{align*}
		H(z):=\sum_{k=-\infty}^{\infty}h(k) z^{-k} \ ,
	\end{align*}
	and we assume that the sum converges for any $z\in \C$ such that $r_1<|z|<r_2$ with $r_1<1<r_2$.
	Besides, we also assume that any entry of the $N$-dimensional vector $H(z)$ is a real-rational function of $z$.
	By the properties of the $\mathcal{Z}$-transform, $H(z)$, along with the convergence domain defined by $r_1$ and $r_2$, uniquely identifies the sequence $h(k)$.
	Moreover, given a vector of rationally related random processes $e(t):=(e_1(t),..,e_{N}(t))^T$, denote for any $t\in \Z$ the random variable
	\begin{align*}
		y_t:=\sum_{k=-\infty}^{\infty}h(k) e(t-k).
	\end{align*}
	Then, we define by $H(z)e$ the related stochastic process such that
	\begin{align*}
		(H(z)e)(t)=y_t \qquad \forall~t\in\Z \ .
	\end{align*}
\end{define}

\begin{define}
	Given a  $N$-dimensional time-discrete, zero-mean, wide-sense stationary random process $e(t):=(e_1(t),..,e_{N}(t))^T$, we define its power spectral density $\Phi_{e}(z)$ as:
	\begin{align*}
		\Phi_{e}(z):=\sum_{\tau=-\infty}^{\infty}R_{e}(\tau)z^{-\tau} \ ,
	\end{align*}
	having a certain domain of convergence $\mathcal{D}\subseteq \C$ in the variable $z$.
	Denoting by $\Phi_{e_i e_j}(z)$ the entry $(i,j)$ of $\Phi_{e}(z)$, it follows that
	\begin{align*}
		\Phi_{e_i e_j}(z):=\sum_{\tau=-\infty}^{\infty}R_{e_i e_j}(\tau)z^{-\tau}.
	\end{align*}
	Besides, if for any $i,j\in \{1,...,N\}$ the power spectral density $\Phi_{e_i e_j}(z)$ exists on the unit circle $|z|=1$ of the complex plane and it is a real-rational function of $z$, we formally write
	\begin{align*}
		\Phi_{e_i e_j}(z)=\frac{A(z)}{B(z)}\qquad \text{for } i,j=1,\ldots, N,
	\end{align*}
	with $A(z),B(z)$ real coefficient polynomials, such that $B(z)\neq 0$ for any $z\in \C, |z|=1$.
	In such a case, we say that $e$ is a vector of rationally related random processes.
\end{define}
%
Finally, let us introduce the following sets:
\begin{align*}
	\mathcal{F}:=
	&\{W(z)| W(z)\text { is a real-rational scalar function}\\
	&\quad\text{of } z\in\C\text{ defined for } |z|=1\}\\
	\mathcal{F}^{m\times n}:=
	&\{W(z)| W(z)\in \C^{m\times n}\text { and any }\\
	&\quad\text{of its entries is in $\mathcal{F}$} \}.
\end{align*}

\begin{prob} \label{pro1}
Consider a set $\mathcal{X}:=\{x_{1},...,x_{n}\}\subset \mathcal{F}e$ of $n$ rationally related processes with zero mean and known (cross)-power spectral densities $\Phi_{x_i x_j}(z)$.
Then, in the above framework the mathematical formulation of the considered problem can be stated as follows:
\begin{align}\label{eq:optimization}
	\min_{\stackrel	{\alpha_{j,1},...,\alpha_{j,m_j}\neq j}
			{W_{j,\alpha_{j,k}}(z)\in \mathcal{F}}
	}
	E\left\{
		x_j-\sum_{k=1}^{m_j}W_{j,\alpha_{j,k}}(z)x_{\alpha_{j,k}}
	\right\} \ ,
\end{align}
where every $W_{j,\alpha_{j,k}}(z)$, with $k=1,...,m_j$, is a possibly non-causal transfer function.
\end{prob}

\begin{remrk}
Fixed any set $\{\alpha_{j,k}\}_{k=1}^{m_j}$, Problem \ref{pro1} is immediately solved by a multiple input Wiener filter. However, the determination of the parameters $\alpha_{j,k}$ makes the problem combinatorial.
\end{remrk}

\section{A geometric interpretation}\label{sec:geometric}

It is possible to give a geometrical interpretation of (\ref{eq:optimization}) by embedding the 
processes $x_1, \ldots, x_M$ in a suitable vector space. This interpretation has the main advantage of giving to the Wiener filter the meaning of a  projective operator in such a space.

\begin{define}
	Let $e=(e_1,...,e_N)^T$ be a vector of $N$ rationally related random processes.
	We define the set $\mathcal{F}e$, as
	\begin{align*}
		\mathcal{F}e:=
		\left\{x=H(z)e~|~H(z)\in \mathcal{F}^{1\times N}\right\}.
	\end{align*}
\end{define}

\begin{propos}
	The ensemble $(\mathcal{F}e,+,\cdot,\R)$ is a vector space.
\end{propos}
\begin{proof}
	Consider $X_1,X_2,X_3\in\overline{\mathcal{F}e}$, $x_1\in X_1$,$x_2\in X_2$, $x_3\in X_3$, and $\alpha_1,\alpha_2\in \Re$.
	\begin{itemize}
		\item	Commutativity for the sum\\
		Consider
		\begin{align}
			x:=x_1+x_2 \in X_1+X_2.
		\end{align}
		Then, since $x=x_2+x_1$, we also have that $x\in X_2+X_1$.
		Since $(\overline{\mathcal{F}e},+,\cdot,\R)$ is a partition, we obtain $X_1+X_2=X_2+X_1$.
		\item	Associativity for the sum\\
		Consider
		\begin{align}
			x_a:=x_1+(x_2+x_3) \in X_1+(X_2+X_3)\\
			x_b:=(x_1+x_2)+x_3 \in (X_1+X_2)+X_3\\
		\end{align}
		Then, since $x_a=x_b$, for the property of a partition set,
		$X_1+(X_2+X_3)=(X_1+X_2)+X_3$.
		\item	Additive identity\\
		The process $x_0(t)=0$ for any $t$ is in $\mathcal{F}e$, because the zero transfer function is in $\mathcal{F}$. Let the set $X_0\in\overline{\mathcal{F}e}$ be the set that constains $x_0$.
		Since $x_1+x_0=x_1$ for any $x_1$, we have that $X_0$ is the identity element for the addition.
		\item	Additive inverse\\
		If $x_1\in \mathcal{F}e$, then also $-x_1\in \mathcal{F}e$, because the transfer function $-1\in \mathcal{F}$.
		\item	Scalar multiplication identity\\
		Let $x_1$ be a process in $X_1$. The scalar $1$ is the multiplication identity. Indeed, we have
		\begin{align}
			1\cdot x_1 = x_1 \in X_1.
		\end{align}
		Thus, it holds that $1\cdot X_1=X_1$.
		\item Associativity of the scalar multiplication\\
		Since we have that $x=\alpha_1(\alpha_2 x_1)=(\alpha_1\alpha_2)x_1$, we also have that  $\alpha_1(\alpha_2 X_1)=(\alpha_1\alpha_2)X_1$.
		\item Distribuitivity of the scalar sum\\
		Since we have $(\alpha_1+\alpha_2)x_1=\alpha_1x_1+\alpha_2x_1$, we also have $(\alpha_1+\alpha_2)X_1=\alpha_1X_1+\alpha_2X_1$
		\item Distribuitivity of the vector sum\\
		Since we have $\alpha_1(x_1+x_2)=\alpha_1x_1+\alpha_1x_2$, we also have $\alpha_1(X_1+X_2)=\alpha_1X_1+\alpha_1X_2$.
		$\hfill\square$
	\end{itemize}
\end{proof}
\begin{define}
	For any $x\in \mathcal{F}e$ we denote the norm induced by the inner product as
	\begin{align*}
		\|x\|:=<x,x>.
	\end{align*}
\end{define}

As shown in details in the Appendix,	the set $\mathcal{F}e$, along with the operation $<\cdot,\cdot>$ is a pre-Hilbert space (with the technical assumption that $x_1$ and $x_2$ are the same processes if $x_1 \sim x_2$).

We provide an ad-hoc version of the Wiener Filter (guaranteeing that the filter will be real rational) with an interpretation in terms of the Hilbert projection theorem.
Indeed, given signals $y, x_1,...,x_n \in \mathcal{F}e$,
the Wiener Filter estimating $y$ from $x:=(x_1,...,x_n)$ can be interpreted as the operator that determines the projection of $y$ onto the subspace $\mathcal{F}x$
\begin{propos}\label{eq: my wiener}
	Let $e$ be a vector of rationally related processes.
	Let $y$ and $x_1,...,x_n$ be processes in the space $\mathcal{F}e$.
	Define $x:=(x_1,...,x_n)^T$ and consider the problem
	\begin{align}\label{eq: cost general wiener}
		\inf_{W \in \mathcal{F}^{1\times n}} \|y-W(z)x\|^2.
	\end{align}
	If $\Phi_{x}(\w)>0$, for all $\w\in[-\pi,\pi]$, the solution exists, is unique and has the form
	\begin{align*}
		W(z)=\Phi_{yx}(z)\Phi_{xx}(z)^{-1}.
	\end{align*}
	Moreover, for any $W'(z)\in \mathcal{F}^{1\times n}x$, it holds that
	\begin{align}\label{eq:perp hilbert projection thm}
		<y-W(z)x,W'(z)x>=0.
	\end{align}
\end{propos}

\begin{proof}
	Observe that, since $q\in X$, the cost function satisfies
	\begin{align*}
		&\|y-W(z)x\|^2=\int_{-\pi}^{\pi}\Phi_{yy}(\w)+W(\w)\Phi_{xx}(\w)W^*(\w)+\\
		&\qquad-\Phi_{xy}(\w)W^{*}(\w)-W(\w)\Phi_{yx}(\w)d\w.
	\end{align*}
	The integral is minimized by minimizing the integrand for all $\w\in [-\pi,\pi]$.
	It is straightforward to find that the minimum is achieved for
	\begin{align*}
		W(\w)=\Phi_{yx}(\w)\Phi_{xx}^{-1}(\w).
	\end{align*}
	Defining the filter $W(z)=\Phi_{x x_{I}}(z)\Phi_{x_{I} x_{I}}(z)^{-1}$ a real-rational transfer matrix is obtained with no poles on the unit circle that has the specified frequency response. Thus $\hat x=W(z)x_I\in X$ minimizes the cost (\ref{eq: cost general wiener}).
	Equation (\ref{eq:perp hilbert projection thm}) is an immediate consequence of the Hilbert projection theorem (for pre-Hilbert spaces) \cite{Lue69}.

\end{proof}

\begin{prob}\label{pro2}
The mathematical formulation of the problem can be seen now as the following:
\begin{align}\label{eq:optimization2}
	\min_{\stackrel	{\alpha_{j,1},...,\alpha_{j,m_j}\neq j}
			{W_{j,\alpha_{j,k}}(z)\in \mathcal{F}}
	}
	\left\|
		x_j-\sum_{k=1}^{m_j}W_{j,\alpha_{j,k}}(z)x_{\alpha_{j,k}}
	\right\|^2 \ ,
\end{align}
\end{prob}

\section{Links with compressive sensing} \label{sec:links}
In this section we highlight the connections between the problem of modeling a network topology and the compressive sensing problem.
Such a connection is possible because of the pre-Hilbert structure constructed in Section \ref{sec:geometric} and Appendix. Indeed, the concept of inner product defines a notion of ``projection'' among stochastic processes and makes it possible to seamlessly import tools developed for the compressive sensing problem in order to tackle that of describing a sparsified topology.\\
In the recent few years sparsity problems have attracted the attention of researchers in the area of Signal Processing. The reason is mainly due to the possibility of representing a signal using only few elements (words) of a redundant base (dictionary).
Applications are numerous, ranging from data-compression to high-resolution interpolation, and noise filtering \cite{DauDev09}, \cite{WipNag09}.\\
There are many formalizations of the problem, but one of the most common is to cast it as 
\begin{align}\label{eq:l0 standard}
	\min_{w} \|x_0-\Psi w \|_2
	\quad \text{subject to} \quad
	\|w\|_0 \leq m \ ,
\end{align}
where $n<p$, $x_0\in \R^{p}$, $\Psi\in\R^{p\times n}$ is a matrix, whose columns represent a redundant base employed to approximate $x_0$ and the ``zero-norm'' (it is not actually a norm)
\begin{align}
	\| w \|_{0}:=|\{i\in \N| w_i\neq 0\}|
\end{align}
is defined by the number of non-zero entries of a vector $w$.
It can be said that $w$ is a ``simple'' way to express $x_0$ as a linear combination of the columns of $\Psi$, where the concept of ``simplicity'' is given by a constraint on the number of non-zero entries of $w$.

For each $j=1,...,n$ define the following sets:
\begin{align}
	\mathcal{W}^{(j)}=\{W(z)\in \mathcal{F}^{1\times n}| W_j(z)=0\} \ ,
\end{align}
where $W_j(z)$ denotes the $j$-th component of $W(z)$.
For any $W\in \mathcal{W}^{(j)}$, define the ``zero-norm'' as
\begin{align*}
	\|W\|_0=\{\text{\# of entries such that } \exists~z\in\C, W_i(z)\neq 0  \}
\end{align*}
and define the random vector
\begin{align}
	x=(x_1, ... , x_n)^T.
\end{align}
Then,  the problem (\ref{eq:optimization}) can be formally cast as
\begin{align}\label{eq:l0 our way}
	\min_{W\in\mathcal{W}_j} \|x_j-W x \|^2
	\quad \text{subject to} \quad
	\|W\|_0 \leq m
\end{align}
which is, from a formal point of view, equivalent to the standard $l_0$ problem as defined in (\ref{eq:l0 standard}).
\section{Solution via suboptimal algorithms}\label{sec:solution via suboptimal}
The problem of modeling network interconnections/complexity reduction we have formulated in this paper is equivalent to the problem of determining a sparse Wiener filter, as explained in the previous section, once a notion of orthogonality is introduced.
This formal equivalence shows how deriving a suitable topology can immediately inherit a set of practical tools already developed in the area of compressing sensing.\\
Here we present, as illustrative examples, modifications of algorithms and strategies, well-known in the Signal Processing community, which can be adopted to obtain suboptimal solutions to the problem of modeling the network interconnections.\\
While formally identical to \eqref{eq:l0 standard}, the problem of a topology reconstruction cast as in \eqref{eq:l0 our way} still has its own characteristics. Since the ``projection'' procedure in \eqref{eq:l0 our way} is given by the estimation of a Wiener filter, it is computationally more expensive than the standard projection in the space of vectors of real numbers.
For this reason greedy algorithms offer a good approach to tackle the problem since speed becomes a fundamental factor. Moreover, since the complexity of the network model is here one of final goal, greedy algorithms are a suitable solution, since they allow one to specify explicitly the connection degree $m_j$ of every node $x_j$. This feature is in general not provided by other algorithms.
As an alternative approach to greedy algorithms we also describe a strategy based on iterated reweighted optimizations as described in \cite{CanWak08}.
\subsection{A modified Orthogonal Least Squares (Cycling OLS)}
Orthogonal Least Squares (OLS) is a greedy algorithm proposed for the first time in \cite{CheBil89} and in many ways it resembles the algorithm of Matching Pursuit developed in \cite{MalZha93}.
It basically consists of iterated orthogonal projections on elements  of a (possibly redundant) base to approximate a given vector.
For the details of this algorithm we remand the reader to \cite{CheBil89}.
However, for the sake of clarity, we reformulate it in terms of our problem.
The initialization occurs at the first step setting the set of the chosen elements of the dictionary to $\Gamma^{(1)}=\emptyset$.
At the $l-$th iteration step, OLS determines the term $\hat x_j^{(l,i)}$ to be added to the reduced dictionary by projecting $x_j$ onto the space generated by $\Gamma^{(l,i)}:=\Gamma^{(l-1)}\cup \{x_i\}$ for any $i\neq j$.
Then $\Gamma^{(l)}$ is defined as the $\Gamma^{(l,i)}$ for which $\|x_j-\hat x_{j}^{(l,i)}\|$ is the smallest and the the algorithm moves to the next iteration step.
The standard OLS goes on at every step introducing a new vector until a stopping condition is met (usually on the norm of the residual $r^k$ or on the number of iterations).\\
We propose an algorithm which derives directly from OLS but it does not increase the number of vectors $x_{\alpha_{j,k}}$ approximating $x_j$ above $m_j$.
The variation from OLS is very simple. At any iteration, given the set of vectors $\Gamma^{(l-1)}$, if it already contains $m_j$ vectors, the algorithm chooses a vector in $\Gamma^{(l-1)}$ to be removed and tries to replace it with another vector in order to improve the quality of the approximation and updates it. If such an improvement is not possible by removing any of the vectors in the current selection, the algorithm stops. The implementation can be described using the following pseudocode.\\
~\\
{\tt Cycling Orthogonal Least Squares: }
\begin{itemize}
	\item[{\tt 0.}] define $x_0:=0$ (null time series) and $c=0$.
	\item[{\tt 1.}] initialize the $m_j$-tuple  $S=(x_0, x_0 ..., x_0)$ and $k=1$
	\item[{\tt 2.}] while $c \leq m_j$
	\begin{itemize}
		\item[{\tt 2a.}] for $i=1,...,n$, $i\neq j$\\
		define $S_i$ as the $m_j$-tuple where $x_i$ replaces the $k$-th element of $S$ and \\
		define $\hat x_j^{(i)}$ as the projection of $x_j$ on to $S_i$
		\item[{\tt 2b.}] $\alpha = \arg\max_{i} \|x_j-\hat x_j^{(i)}\|$
		\item[{\tt 2c.}] if $x_{\alpha} = S[k]$ then $c=c+1$
		\item[{\tt 2d.}] else $S[k]=x_\alpha$, $c=1$, $k=k$ mod $m_j$, $k=k+1$
	\end{itemize}
	\item[{\tt 3.}] return $S$
\end{itemize}
The reason of our modification is simple. COLS implements a coordinate descent guaranteeing that the number of non-zero components of the solution does not exceed $m_j$. Once such a limit has been reached, it tries to improve the quality of the approximation without reducing the sparsity of the current solution.
\subsection{Solution via Re-Weighted Least Squares (RWLS)}
Another possible approach to ``encourage'' sparse solutions is provided by reweighted minimization algorithms as proposed in \cite{CanWak08} and \cite{DauDev09}.
A comparison between reweighted norm-$1$ and norm-$2$ methods is performed in \cite{WipNag09}. We consider only reweighted least squares, because such an algorithm is easier to implement, but the intuition behind the two techniques is basically the same.\\
Using Parseval theorem, Problem \ref{eq:optimization}, can be formulated as
\begin{align}\label{eq:optimization frequency}
	\min_{\stackrel	{\alpha_{j,1},...,\alpha_{j,m_j}\neq j}
			{W_{j,\alpha_{j,k}}(z)\in \mathcal{F}}
	}
		\int_{-\pi}^{\pi}\Phi_{[x_j-\sum_{k=1}^{m_j}W_{j,\alpha_{j,k}}(\w)x_{\alpha_{j,k}}]}(\w)d\w.
\end{align}
Consider the following convex variation of the problem
\begin{align}\label{eq:optimization l2 weighted}
	& \min_{W_{j,k}(z)\in \mathcal{F}}
	\left\{
		\int_{-\pi}^{\pi}\Phi_{[x_j-\sum_{k\neq j}W_{j,k}(\w)x_{k}]}(\w)
	\right\}\\
	& \text{subject to} \nonumber\\
	& \sum_{k=1}^{n}\int_{-\pi}^{\pi}\mu_k W_{j,k}^{*}(\w)W_{j,k}(\w)d\w\leq 1 \nonumber
\end{align}
where $\mu_k\in\R^{n}$ is a set of weights for the filters $W_{j,k}(z)$.
Using the compact notation introduced in Section~\ref{sec:links}, we can equivalently write
\begin{align}\label{eq:optimization l2 reweighted}
	\min_{W\in \mathcal{W}_j}\|x_{j}-Wx\|^2
	\quad \text{subject to} \quad
	\|W\|^2_{\mu}\leq 1
\end{align}
where, for a vector $\mu=(\mu_1,...,\mu_n)$, we define
\begin{align*}
	\|W\|^{2}_{\mu}:=\frac{1}{m_j}\sum_{1}^{n}
		\int_{-\pi}^{\pi}\mu_k W_{j,k}^{*}(\w)W_{j,k}(\w)d\w\leq 1.
\end{align*}
Let us assume that the $\alpha_{j,k}$'s and the relative $W_{\alpha_{j,k}}$ solving (\ref{eq:optimization}) are known. Technically, we could set 
\begin{align}\label{eq: l2 good weights}
	\mu_l:=\frac{1}{m_j}\left(\int_{-\pi}^{\pi} W_{l}^{*}(\w)W_{l}(\w)d\w\right)^{-1} \ ,
\end{align}
if $l=\alpha_{j,k}$ for some $k=1,...,m_j$ and $\mu_l=+\infty$ otherwise. With such a choice of weights, the two problems \eqref{eq:optimization} and \eqref{eq:optimization l2 reweighted} would be equivalent since they would provide the same solutions.
However, Problem \eqref{eq:optimization l2 reweighted} has the advantage of being convex.
Of course, the values $\alpha_{j,k}$ are not a-priori known, thus it is not possible  to evaluate \eqref{eq: l2 good weights}.
An iterative approach has been proposed making use of the intuition that we have just formulated to estimate the weights \eqref{eq: l2 good weights}.\\
~\\
{\tt Reweighted Least Squares: }
\begin{itemize}
	\item[{\tt 0.}] For all $x_j$
	\item[{\tt 1.}] initialize the weight vector $\mu:=0$
	\item[{\tt 2.}] while a stop criterion is met
	\begin{itemize}
		\item[{\tt 2a.}] solve the convex problem
		\begin{align*}
			\min_{W\in \mathcal{W}_j}\|x_{j}-Wx\|^2
			\quad \text{subject to} \quad
			\|W_j\|^2_{\mu}\leq 1
		\end{align*}
		\item[{\tt 2b.}] compute the new weigths
		\begin{align*}
			\mu_{k}=\frac{1}{m_j}
			\int_{-\pi}^{\pi}\|W_{j}(\w)\|d\w
		\end{align*}
	\end{itemize}
	\item[{\tt 3.}] return all the $W_j$'s.
\end{itemize}
At any iteration the convex relaxation of the problem is solved and new weights are computed as a functions of the current solution.
When a stopping criterion is met (usually on the number of iterations), the final solution can be obtained by selecting the $m_j$ largest entries of each $W_j$.
\section{Applications and examples} \label{sec:examples}
In this section we report numerical results obtained implementing the algorithms described in the previous section.
In order to evaluate the performances provided by the two algorithms (COLS and RWLS) we have considered a network of~$20$ nodes as represented in Figure~\ref{fig:results}(true) .
In the graph every node $N_j$ describes a stochastic process $x_j$, while every directed arc form a node $N_i$ to a node  $N_j$ represents a transfer function $H_{ji}(z)\neq 0$ that has been randomly selected from a class of causal FIR filters of order~$5$.
The absence of such an arc implies that $H_{ji}(z)=0$.
Thus, each process $x_j$ follows the dynamics
\begin{align}
	x_j=e_j+\sum_{i\neq j}H_{ji}(z)x_i \ .
\end{align}
Every node signal is also implicitly considered affected by an additive white Gaussian noise $e_j$ such that the Signal to Noise Ratio (SNR) is~$4$ (a very noisy scenario).
All the noise processes are independent from each other.
The network has been simulated for~$2000$ steps obtaining~$20$ time series. The time series have been employed to estimate a non-causal FIR approximation of the Wiener Filters of order $21$ both in the COLS and in the RWLS algorithm. 
In Figure~\ref{fig:results} we  report the results of the identification.
Using a global search the global minima of (\ref{eq:l0 standard}) provide the topologies in Figure~\ref{fig:results}~(reduced 2) and  Figure~\ref{fig:results}~(reduced 3) for the case $m_j=2$ and $m_j=3$ for each node respectively.
In Figure~\ref{fig:results}~(no reduction) we report the topology obtained with no constraint on the maximum number of edges: a small threshold has been introduced to remove any edge $(N_i, N_j)$ associated with $W_{ji}(z)\simeq 0$.
In the second row of graphs we have the results for COLS for the cases $m_j=1$, $m_j=2$ and $m_j=3$. In Figure~\ref{fig:results}~(COLS variable) we report the result given by the implementation of a strategy to automatically determine the number of edges: the number of edges is increased only if gives a reduction of~$20\%$ of the residual error.
In the third row of graphs, we report the analogous results for the RWLS algorithm.
\begin{figure*}
	\begin{tabular}{cccc}
		\includegraphics[width=0.22\textwidth]{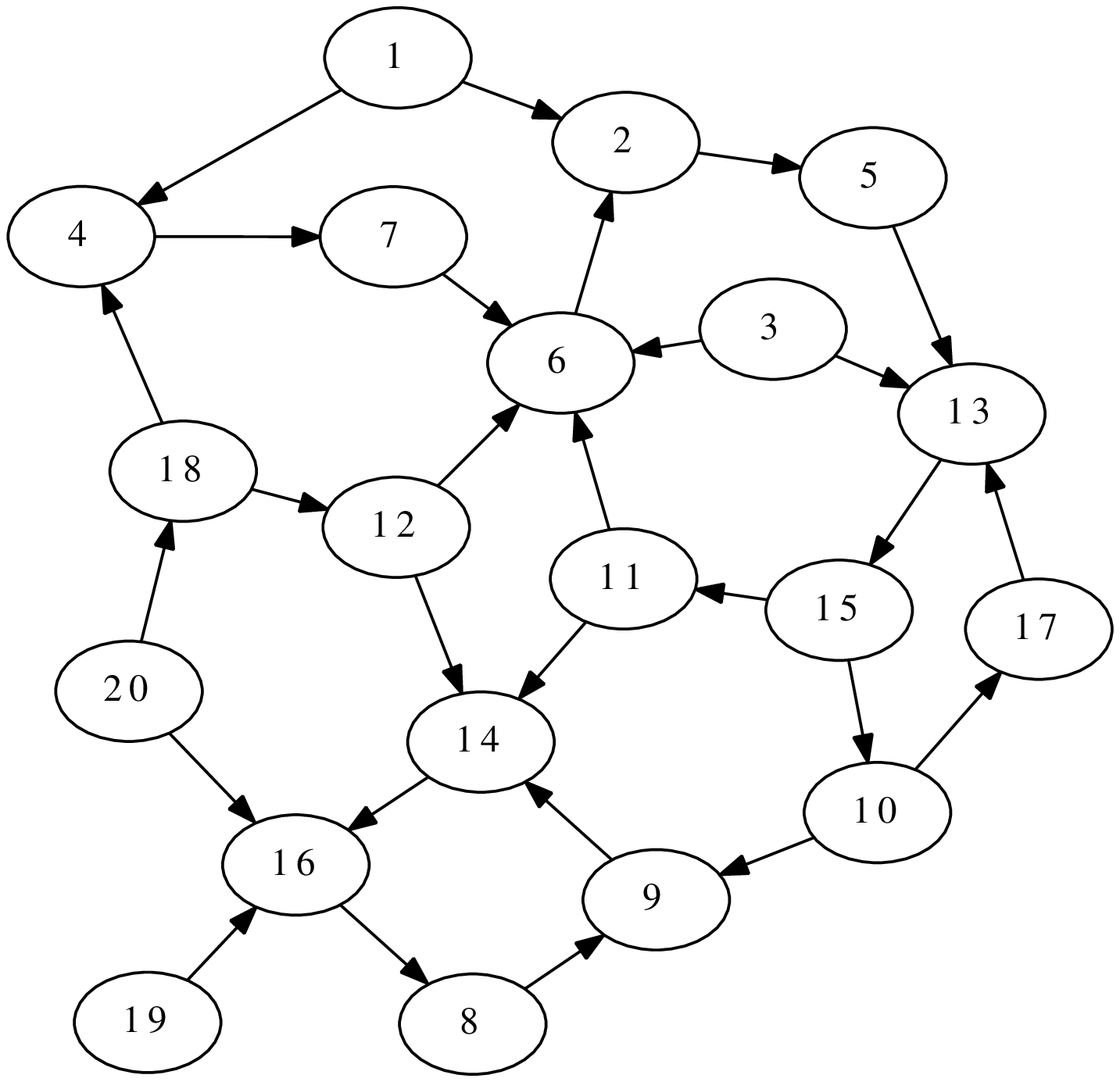} &
		\includegraphics[width=0.22\textwidth]{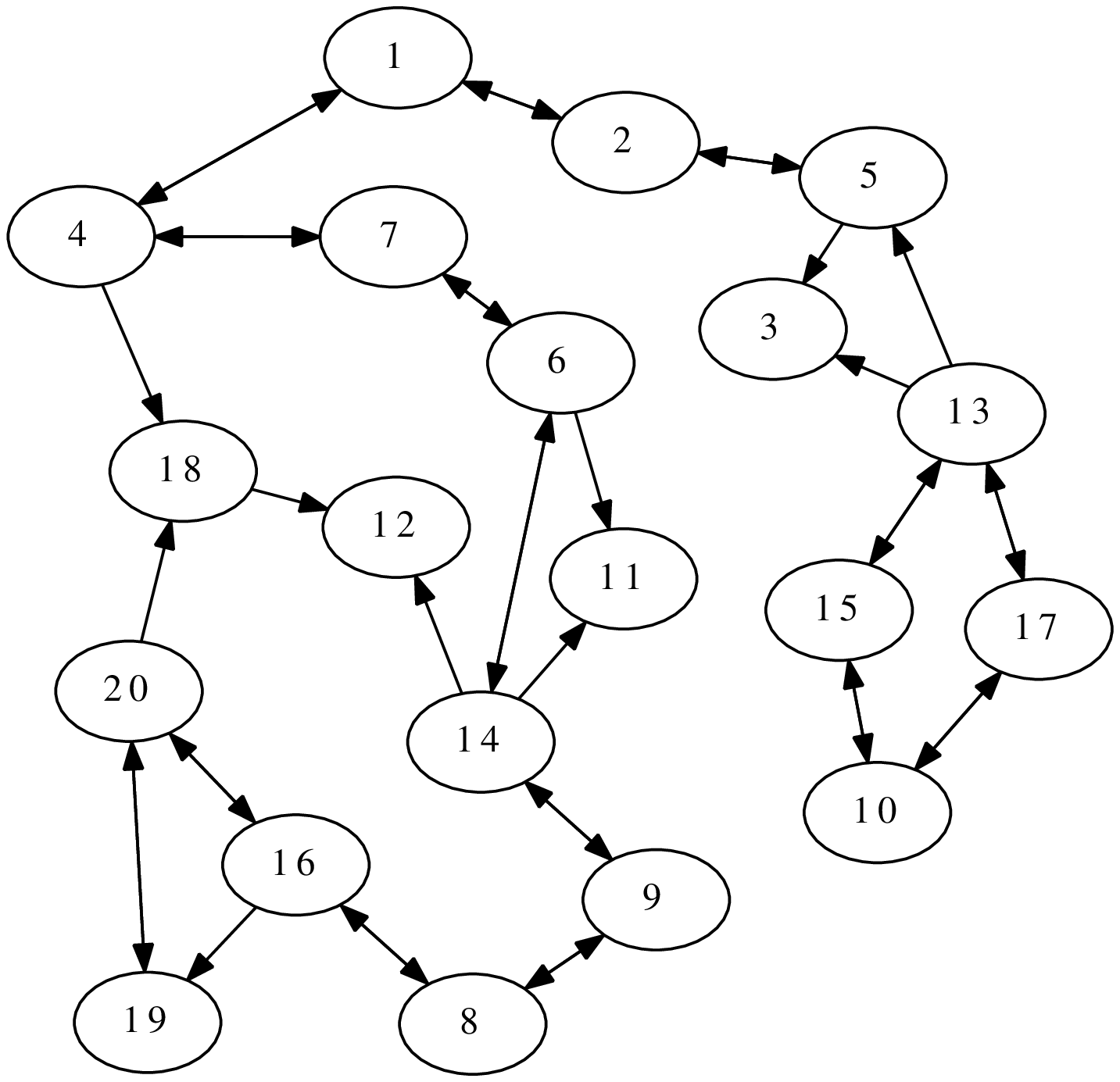} &
		\includegraphics[width=0.22\textwidth]{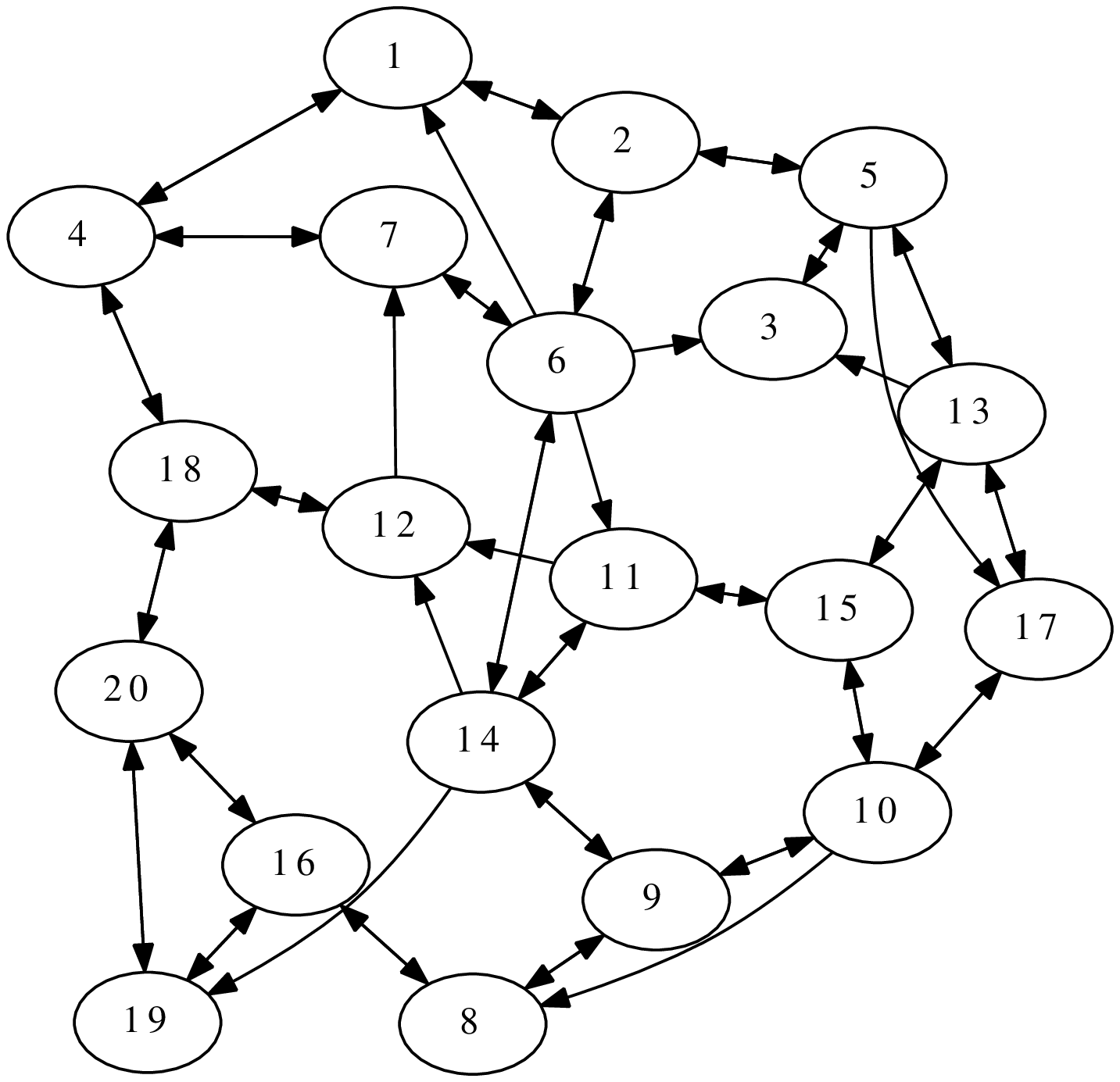} &
		\includegraphics[width=0.22\textwidth]{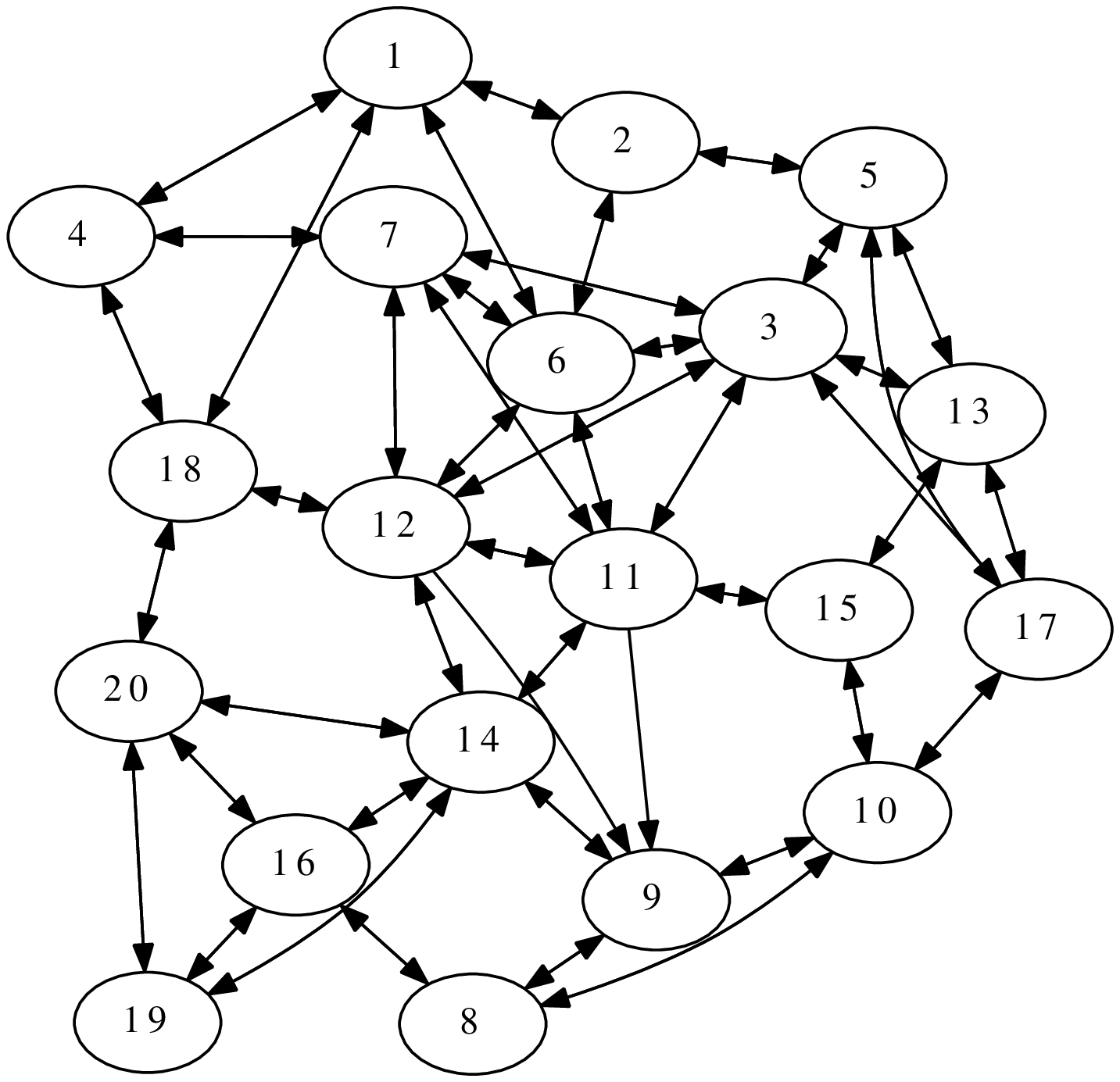}\\
		(true) & (reduced 2) & (reduced 3) & (no reduction)\\
		\includegraphics[width=0.22\textwidth]{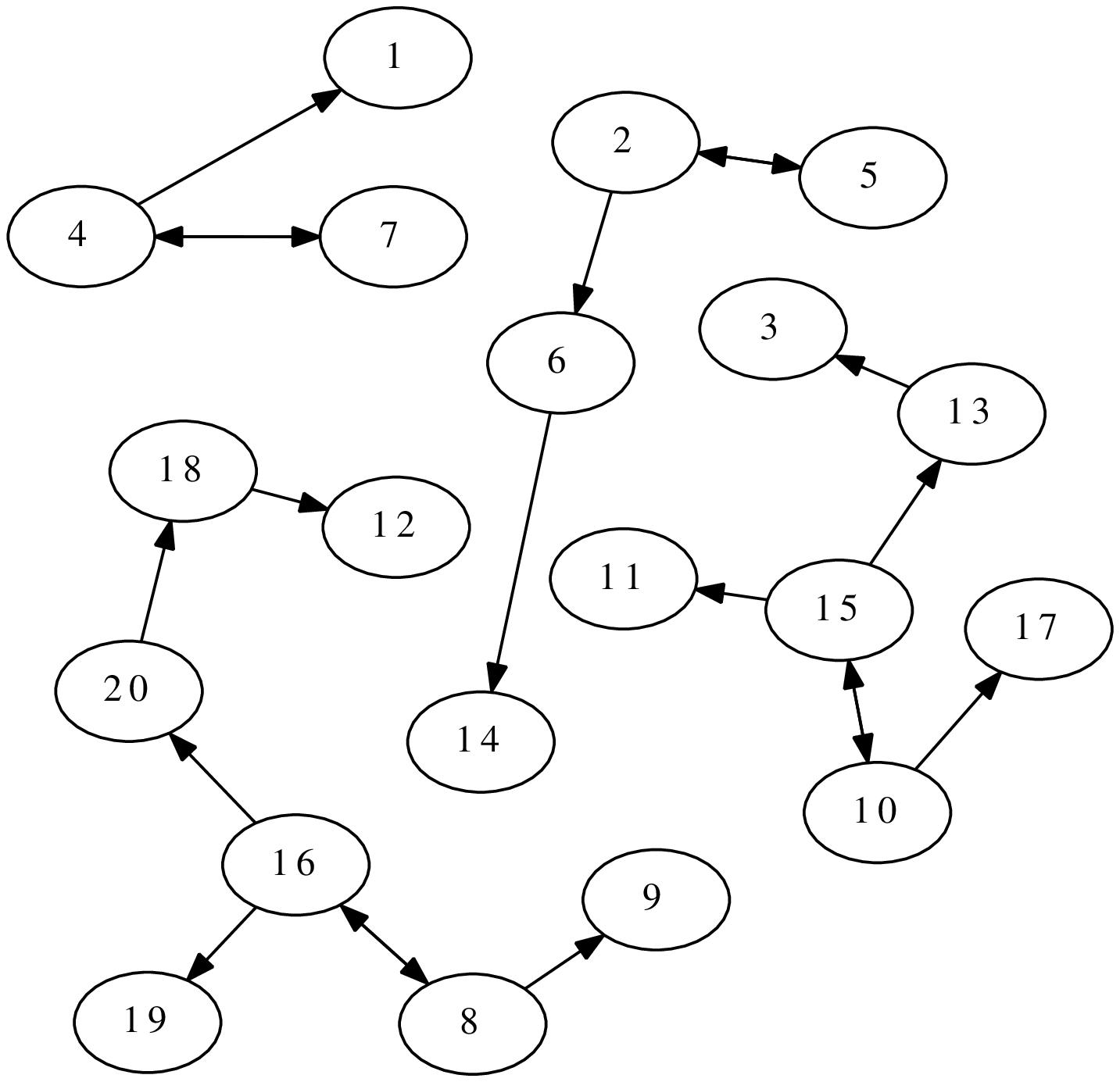} &
		\includegraphics[width=0.22\textwidth]{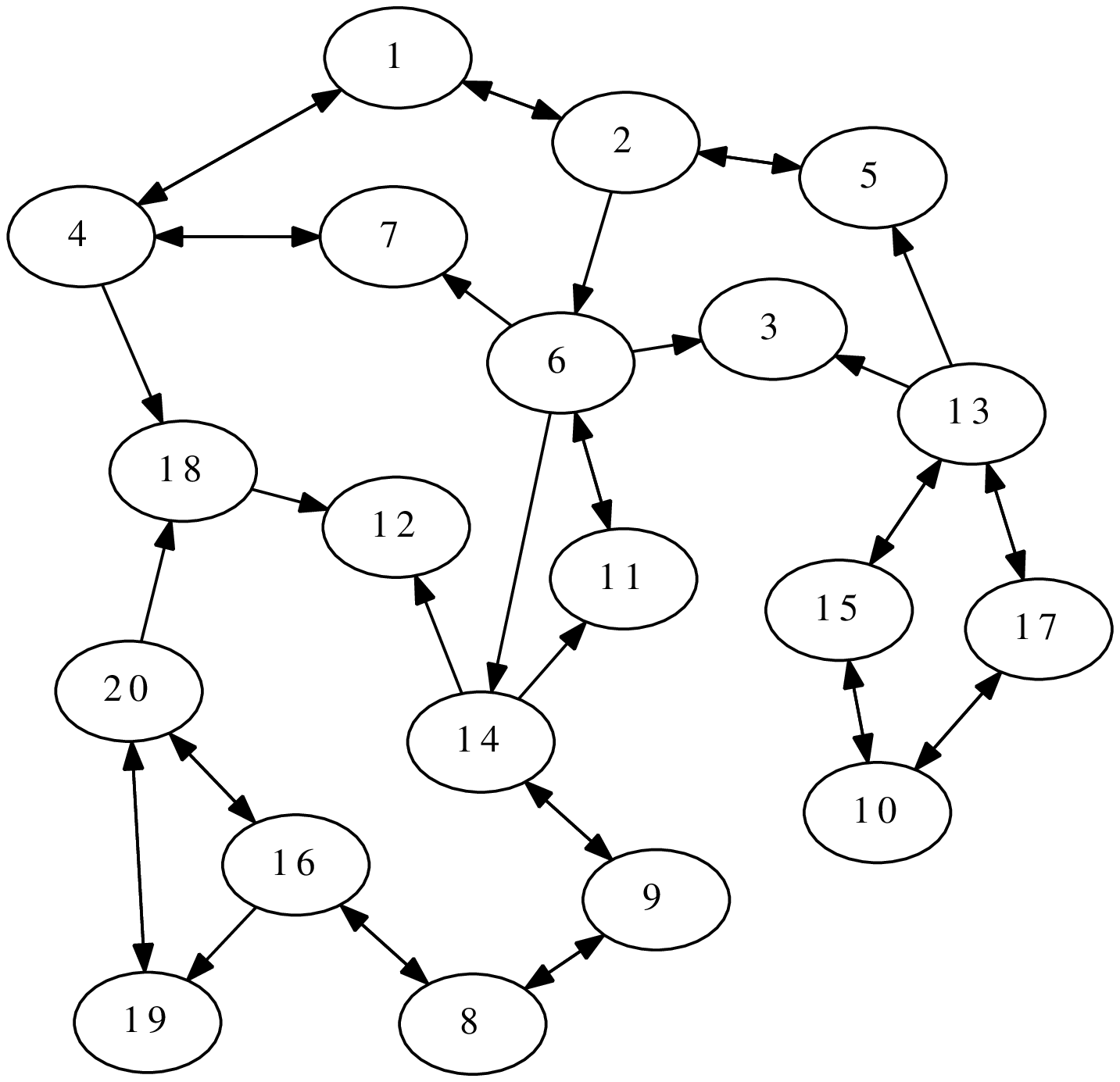} &
		\includegraphics[width=0.22\textwidth]{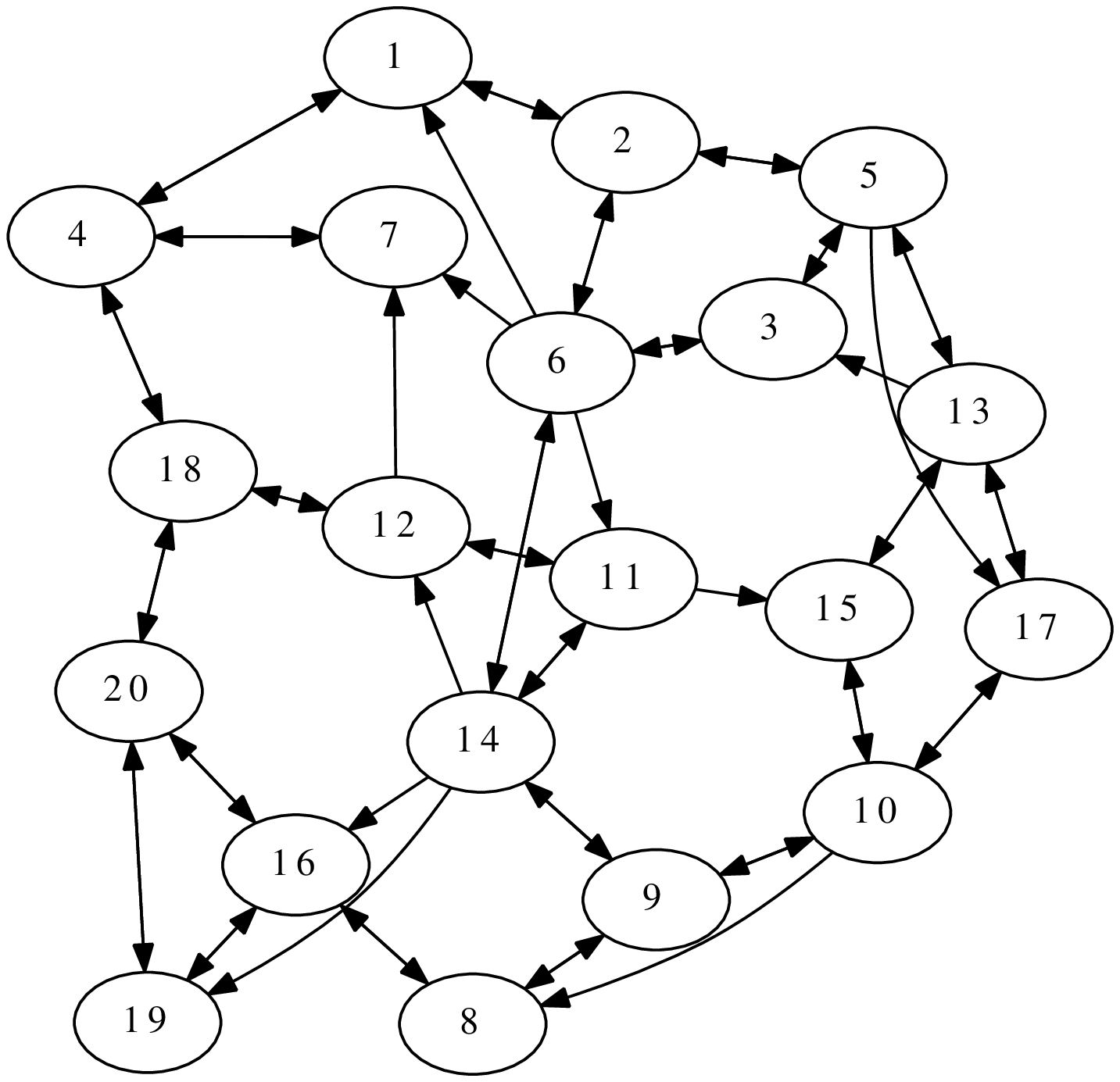} &
		\includegraphics[width=0.22\textwidth]{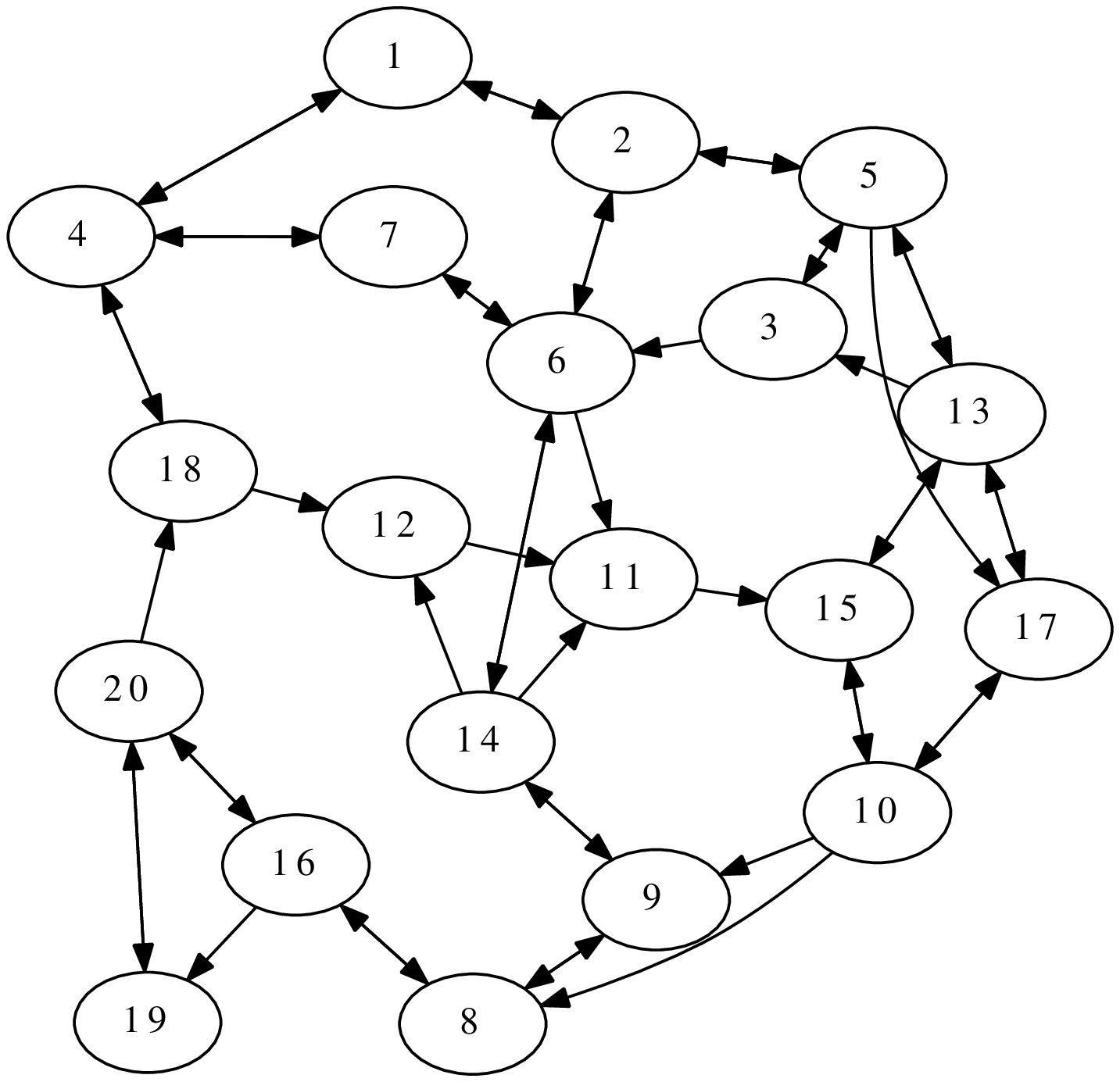}\\
		(COLS 1) & (COLS 2) & (COLS 3) & (COLS variable)\\
		\includegraphics[width=0.22\textwidth]{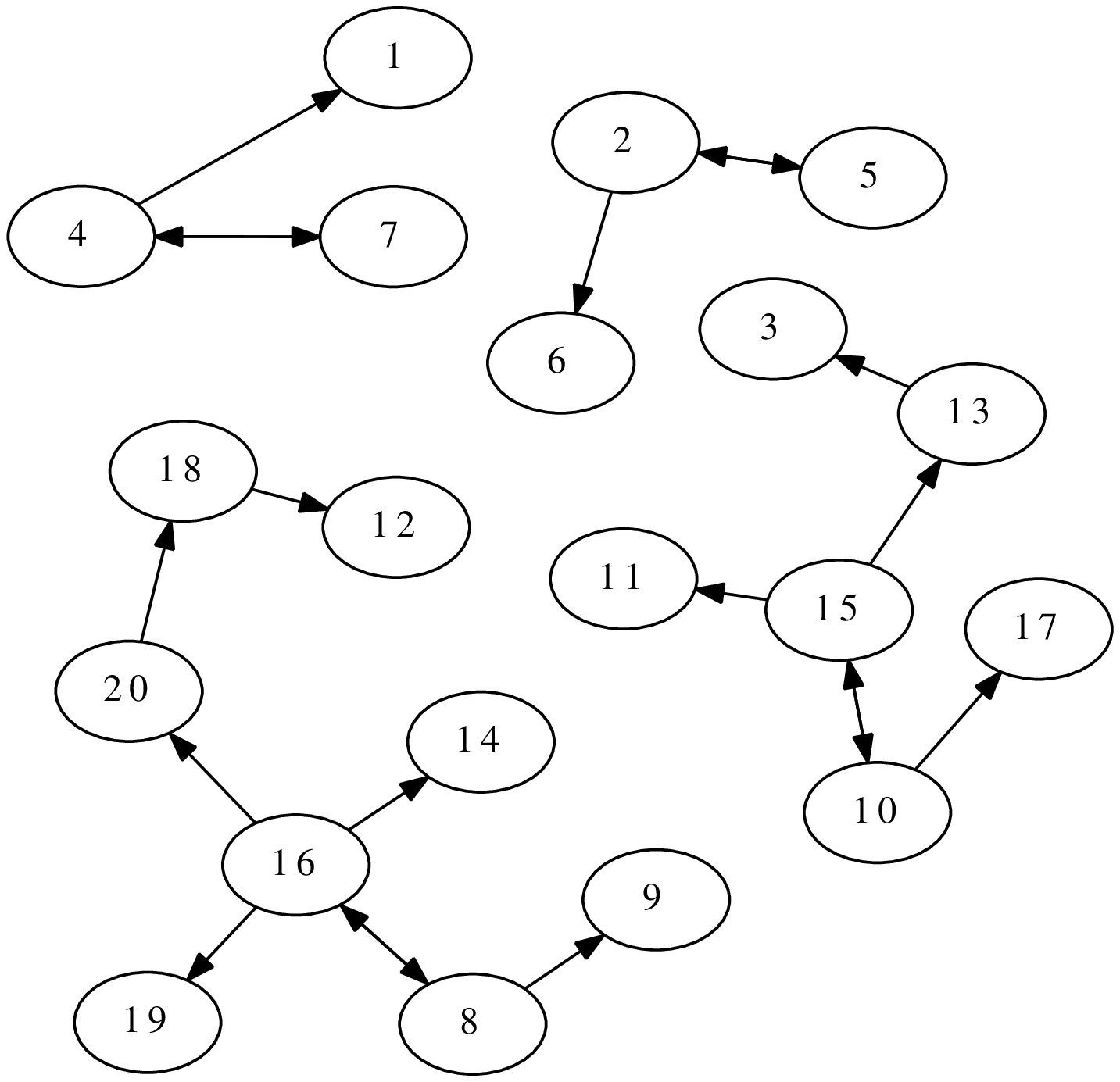} &
		\includegraphics[width=0.22\textwidth]{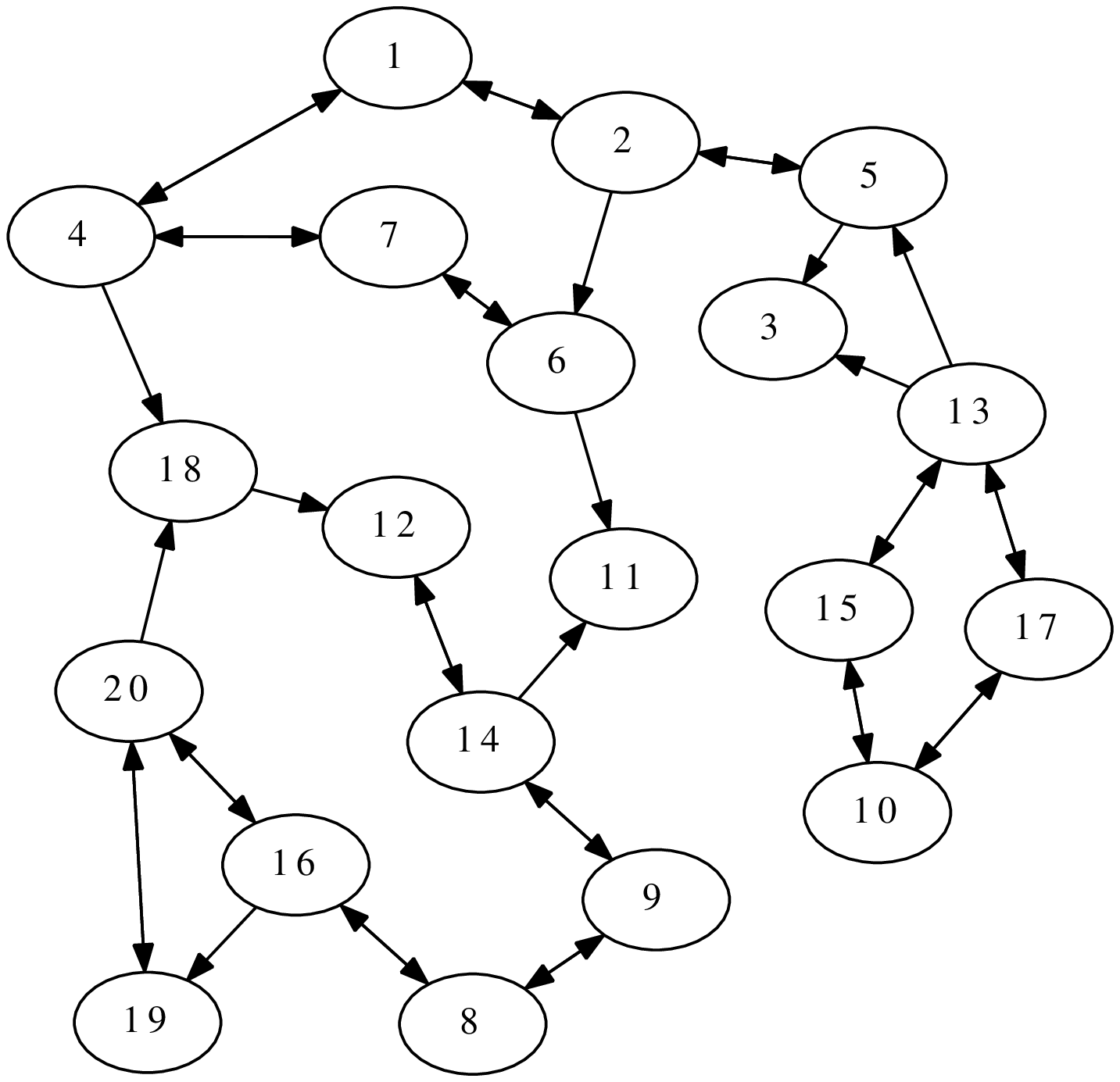} &
		\includegraphics[width=0.22\textwidth]{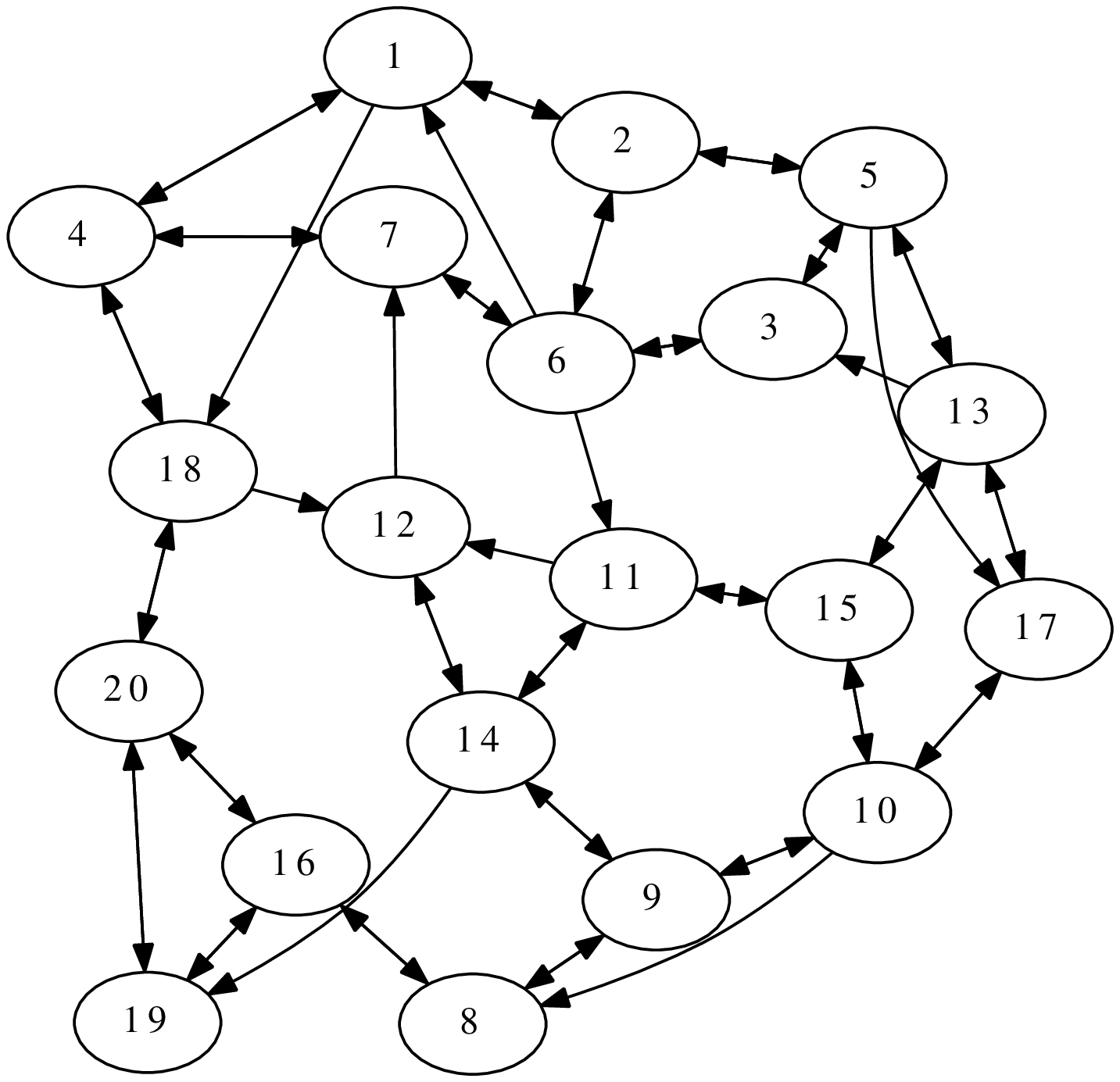} &
		\includegraphics[width=0.22\textwidth]{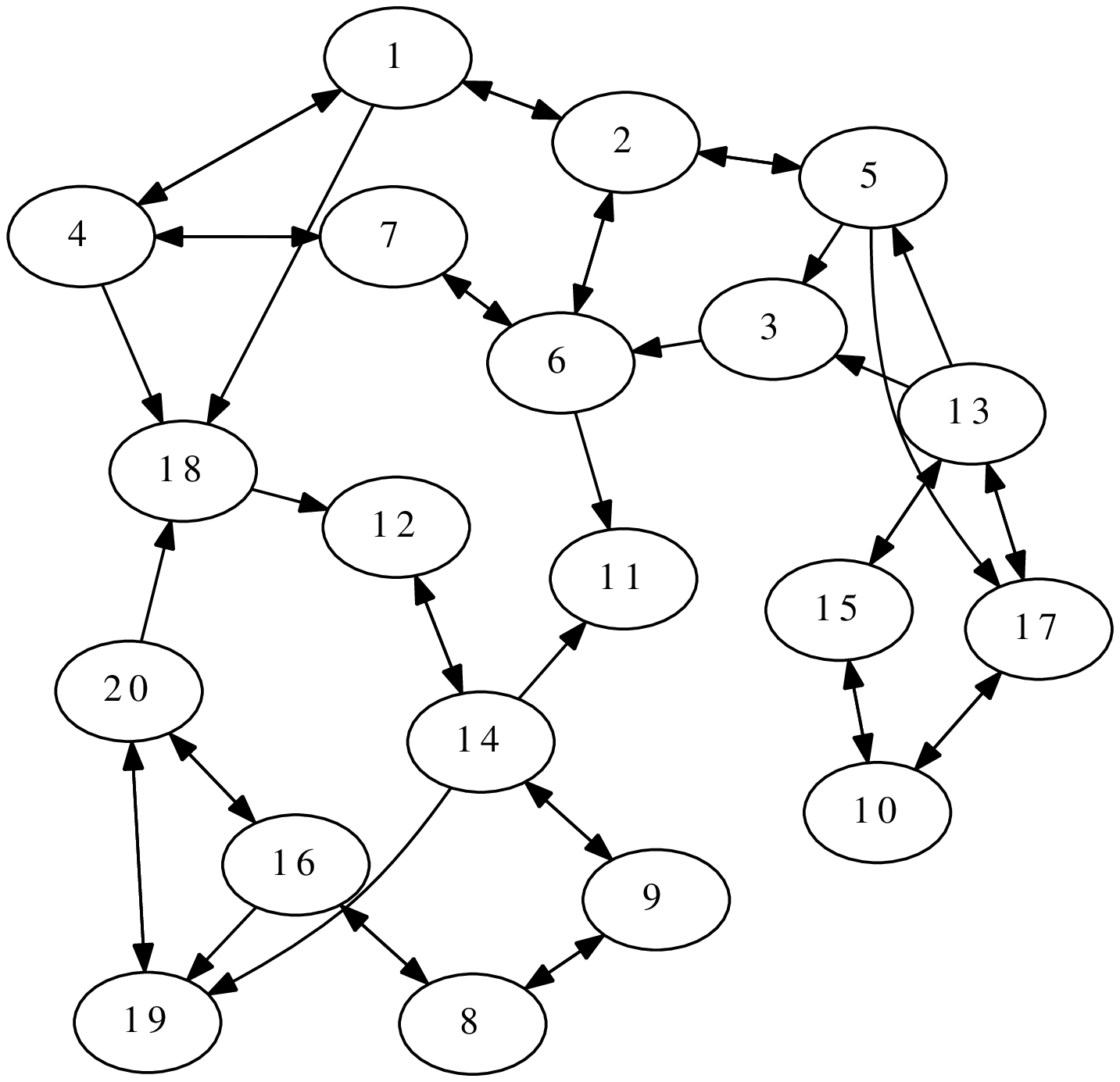}\\
		(RWLS 1) & (RWLS 2) & (RWLS 3) & (RWLS variable)\\
	\end{tabular}
	\caption{The actual network topology (true); the topology obtained by a global minimization of (\ref{eq: cost general wiener}) with $m_j=2$ for every node (reduced 2); the topology with $m_j=3$ for every node (reduced 3); the topology with no constraint on $m_j$ but with a ``small'' threshold imposed on a norm of the Wiener Filters to avoid a complete graph (no reduction); the topologies obtained used the COLS suboptimal approach with $m_j=1$ (COLS 1), $m_j=2$ (COLS 2), $m_j=3$ (COLS 3), and a self-adjusting strategy for $m_j$ (COLS variable): a link is introduced if gives at least a reduction of~$20\%$ of the residual error; the topologies obtained by RWLS after $10$ iterations and keeping only the $m_j$ filters with largest norm (RWLS 1), (RWLS 2), (RWLS 3), or a self-adjusting strategy (RWLS variable).
	\label{fig:results}}
\end{figure*}
\subsection{An application to real data: currency exchange rates}
In this section we also present the results obtained by applying our technique to real data.
We have considered the daily exchange rate of~$22$ selected currencies (reported in Table~\ref{tab:currencies}) from the last~$7$ years providing $1715$ samples for any of the time series.
The missing data (the exchange rate on Saturdays and Sundays) have been interpolated (by cubic splines) such that a total number of~$2400$ daily points have been obtained for our analysis.
\begin{table}[tb]
	\centering
	\begin{tabular}{|c|c|c|}
		\hline
		Name & Code & Country \\
		\hline \hline
	 	Australian Dollar	&AUD	&Australia\\ \hline 
	 	Brazil Real		&BRL	&Brazil\\ \hline 
	 	Canadian Dollar		&CAD	&Canada\\ \hline 
		Chinese Renminbi	&CNY	&China\\ \hline
		Danish Krone		&DKK	&Denmark\\ \hline
	 	Euro			&EU	&European Union\\ \hline 
		British Pound		&GPB	&Great Britain\\ \hline
		Hong Kong Dollar	&HKD	&Hong Kong\\ \hline
		Indian Rupee		&INR	&India\\ \hline
	 	Japanese Yen		&JPY	&Japan\\ \hline 
		South Korean Won	&KRW	&South Korea\\ \hline
		Sri Lankan rupee	&LKR	&Sri Lanka\\ \hline
		Mexican Peso		&MXN	&Mexico\\ \hline
		Malaysian Ringgit	&MYR	&Malaysia\\ \hline
		Norwegian Krone		&NOK	&Norway\\ \hline
		New Zealand Dollar	&NZD	&New Zealand\\ \hline
		Swedish Krona		&SEK	&Sweden\\ \hline
		Singapore Dollar	&SGD	&Singapore\\ \hline
		Thai Baht		&THB	&Thailand\\ \hline
		Taiwanese Dollar	&TWD	&Taiwan\\ \hline
	 	American Dollar		&USD	&United States of America\\ \hline 
		South African Rand	&ZAR	&South Africa\\ \hline
	\end{tabular}
	\caption{List of the currencies considered in the analysis. \label{tab:currencies}}
\end{table}
The Cycling OLS algorithm has been applied on the logarithmic returns of the time series (a standard procedure in Finance,) with order 1,2 e 3 and the estimated topologies are depicted in Figure~\ref{fig:real}~a, Figure~\ref{fig:real}~b and Figure~\ref{fig:real}~c.
\begin{figure*}
	\centering
	\begin{tabular}{ccc}
	\includegraphics[width=0.31\textwidth]{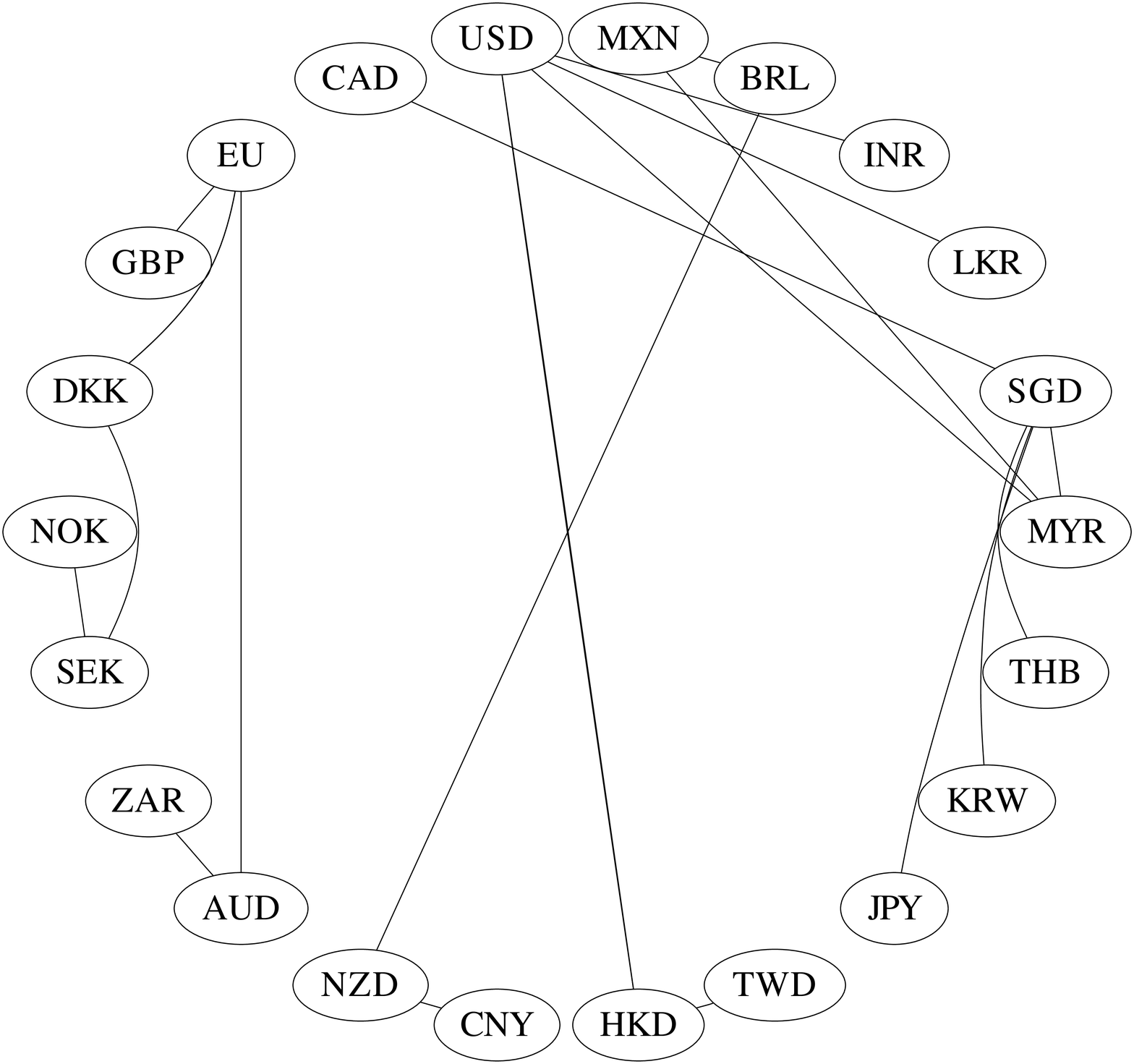} &
	\includegraphics[width=0.31\textwidth]{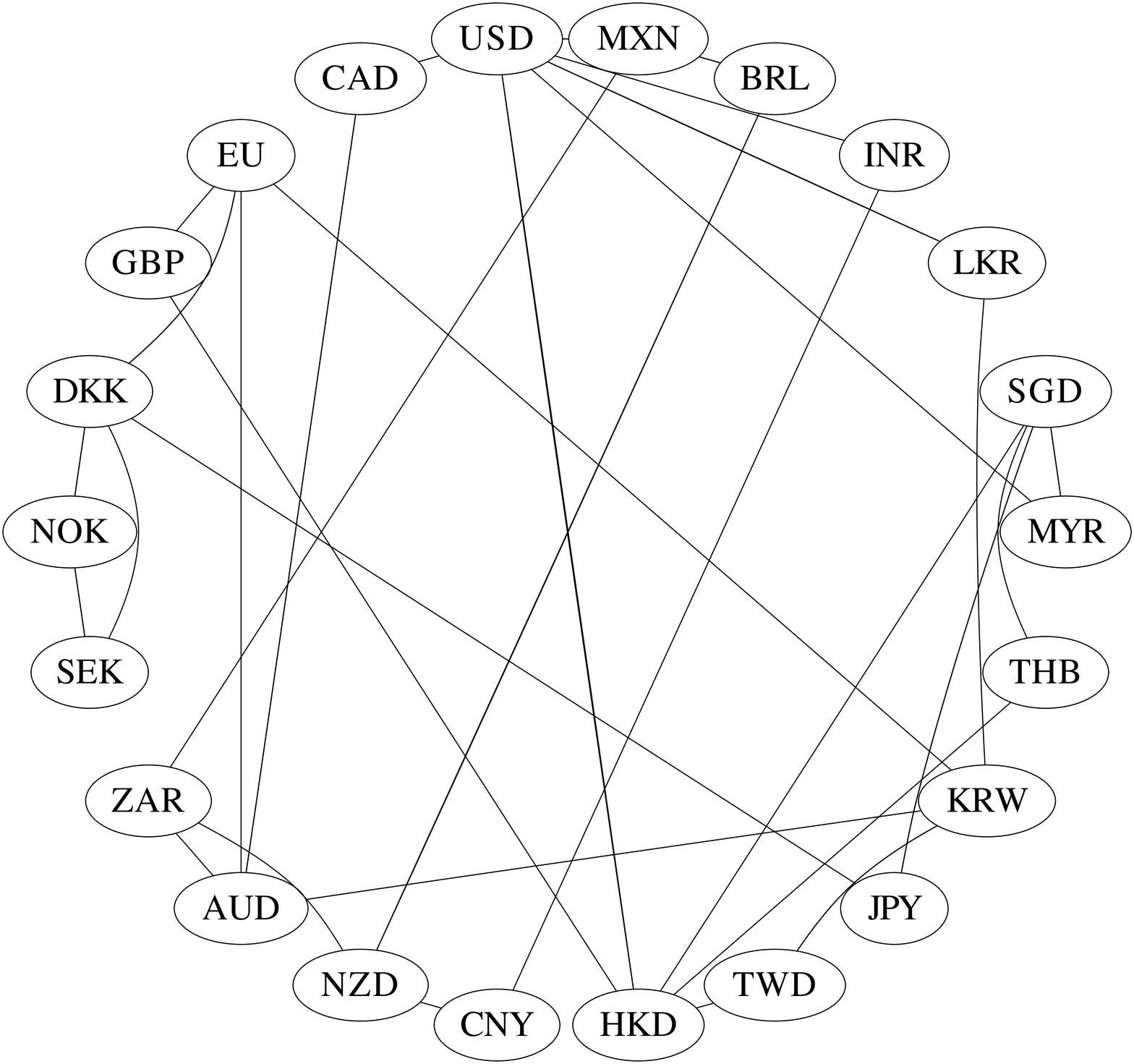} &
	\includegraphics[width=0.31\textwidth]{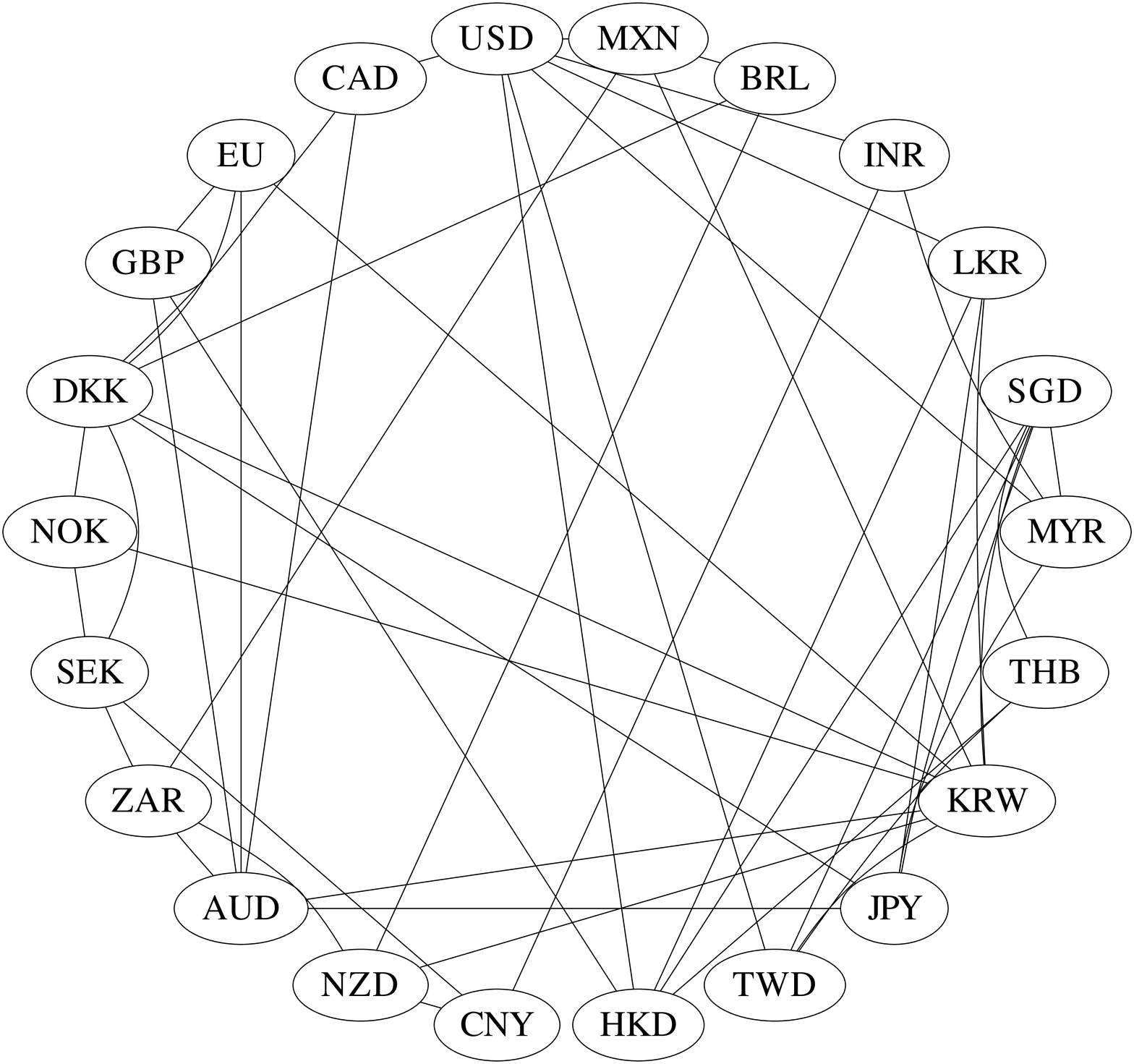}\\
	(a) & (b) & (c)
	\end{tabular}
	\caption{The reconstructed topologies obtaining applying the Cycling OLS to the exchange rate time series of the $22$ selected currencies.
	\label{fig:real}}
\end{figure*}
\section{Final Remarks}
We have formulated the problem of deriving a link structure from a set of time series, obtained by sampling the output of as many interconnected dynamical systems.
Every time series is represented as a node in a graph and their dependencies as connecting edges.
The approach we follow in determining the graph arcs relies on (linear) identification techniques based on an ad-hoc version of the Wiener Filter (guaranteeing that the filter will be real rational) with an interpretation in terms of the Hilbert projection theorem.
If a time series $X_i$ turns outs ``useful'' to model the time series $X_j$, then the directed arc $(i,j)$ is introduced in the graph.
In order to modulate the complexity of the final graph, a maximum number $m_j$ of arcs pointing at $X_j$ is assumed and a cost function is minimized to find the most appropriate arcs.
The problem has a similar formulation and strong connections with the problem of compressing sensing, which has been widely studied in the last few years.
Such a connection is possible because of the pre-Hilbert structure we have  constructed.
Indeed, the concept of inner product defines a notion of ``projection'' among stochastic processes and makes it possible to seamlessly import tools developed for the compressive sensing problem in order to tackle the problem of modeling a network topology.\\
The problem of topology reconstruction/complexity reduction is equivalent to the problem of determining a sparse Wiener filter as explained.
However, since an optimization problem must be solved for any single node, we consider the application of suboptimal solutions.
In particular, we introduce a suboptimal greedy algorithm obtained as a modification of the Orthogonal Least Squares (COLS) and an alternative approach based on iterated ReWeighted Least Squares (RWLS).
By the comparison of the two algorithms on numerical data we have shown the effectiveness of the proposed solutions.
Note that in the present paper no absolute error metric
is provided. \\
Future work will investigate some measure criteria to judge the
performance of the algorithm. For instance, as a starting method we could 
count the correct identified links and the wrong ones, when the underlying topology is known, to define the "more accurate" topology and  extend such a measure of accuracy in some norms to  the unknown topology case.

\bibliographystyle{plain}        

\bibliography{topident,l0problem,control}

\section{Appendix}
We provided hereafter the  definitions and  propositions which are additionally needed for the construction of the pre-Hilbert space.

\begin{define}
	Given two time-discrete scalar, zero-mean, wide-sense jointly stationary random processes $x_1(t)$ and $x_2(t)$, we write that $x_1 \sim x_2$
	if and only if, for any $t\in \Z$, $E[(x_2(t) - x_1(t))^2]=0$, that is $x_1(t)\stackrel{a.s.}{=}x_2(t)$ (that is $x_1(t)=x_2(t)$ almost surely for any $t \in \Z$).
\end{define}

\begin{propos}
	The relation $\sim$ is an equivalence relation on any set $X$ of zero-mean time-discrete wide-sense jointly stationary scalar random processes defined on the time domain $\Z$.
\end{propos}

\begin{propos}
	Let $x_1:=H^{(1)}(z)e$, $x_2:=H^{(2)}(z)e$ be two elements of $\mathcal{F}e$. Then $x_1, x_2$ are scalar, zero-mean, wide-sense jointly stationary random processes with rational power cross-spectral densities having no poles n the set $\{z\in\C|~|z|=1\}$.
\end{propos}
\begin{proof}
	The processes $x_1$ and $x_2$ are scalar by the definition of 
	$\mathcal{F}e$.
	Since $H^{(1)}(z), H^{(2)}(z)\in \mathcal{F}^{1\times n}$, they are real-rational and defined on the unit circle, and, as a consequence, they admit a unique representation in terms of bilateral $\mathcal{Z}$-transform
	\begin{align}
		&H^{(1)}(z)=\sum_{k=-\infty}^{+\infty}h^{(1)}_{k}z^{-k}\\
		&H^{(2)}(z)=\sum_{k=-\infty}^{+\infty}h^{(2)}_{k}z^{-k},
	\end{align}
	with $h^{(1)}_{k}, h^{(2)}_{k}\in \R^{1\times n}$ for $k\in \Z$, such that the convergence is guaranteed on the unit circle $|z|=1$.\\
	First, let us evaluate the mean of $x_1(t)$, that is
	\begin{align*}
		& E[x_1(t)]=E\left[
			\sum_{k=-\infty}^{\infty} h^{(1)}_k e(t-k)
			\right]= \\
			&=\sum_{k=-\infty}^{\infty} h^{(1)}_k E[e(t-k)]=
					H^{(1)}(1)E[e(0)]=0.
	\end{align*}
	Thus, it does not depend on the time $t$.
	Analogously $E[x_2(t)]=0$.
	Now, let us evaluate the cross-covariance function
	\begin{align*}
		&R_{x_1x_2}(t, t+\tau):=E[x_1(t)x_2(t+\tau)^T]=\\
		&=E\left[\left(\sum_{k=-\infty}^{\infty} h^{(1)}(k)e(t-k)\right)
			\left(\sum_{l=-\infty}^{\infty} e^T(t+\tau-l)  (h^{(2)}(l))^T\right)
		\right]=\\
		&=E\left[\sum_{k=-\infty}^{\infty}\sum_{l=-\infty}^{\infty} h^{(1)}(k)e(t-k)
			e^T(t+\tau-l)  (h^{(2)}(l))^T
		\right]=\\
		&=\sum_{k=-\infty}^{\infty}\sum_{l=-\infty}^{\infty} h^{(1)}(k)E[e(t-k)e^T(t+\tau-l)](h^{(2)}(l))^T=\\
		&=\sum_{k=-\infty}^{\infty}\sum_{l=-\infty}^{\infty}h^{(1)}(k)R_{e}(\tau-l+k)(h^{(2)}(l))^T
		=R_{x_1x_2}(0,\tau).
	\end{align*}
	Thus, the cross-covariance does not depend on the time $t$ and, abusing notation, it is possible to define
	\begin{align}
		R_{x_1 x_2}(\tau):=R_{x_1 x_2}(0, \tau)
	\end{align}
	Moreover, if we evaluate the bilateral $\mathcal{Z}$-transform of $R_{x_1x_2}(\tau)$, we have
	\begin{align*}
		&\Phi_{x_1 x_2}(z):=\sum_{\tau=-\infty}^{\infty}R_{x_1x_2}(\tau)z^{-\tau}=\\
		&=\sum_{\tau=-\infty}^{\infty}
			\sum_{k=-\infty}^{\infty}
				\sum_{l=-\infty}^{\infty}h^{(1)}(k)R_{e}(\tau-l+k)(h^{(2)}(l))^T z^{-\tau-l+k}z^{-k}z^l=\\
		&=H^{(1)}(z)\Phi_{e}(z)H^{(2)}(z^{-1}).
	\end{align*}
	$\hfill\square$
\end{proof}

\begin{propos}
	Given a rationally related vector $e$, the set $\mathcal{F}e$ is closed with respect to addition, transformation by $H(z)\in \mathcal{F}$ and multiplication by scalar $\alpha\in\Re$. Moreover, it holds that, for $x_1=H^{(1)}(z)e \in \mathcal{F}e$ and $x_2=H^{(2)}(z)e \in \mathcal{F}e$,
	\begin{align*}
		&H^{(1)}(z)e+H^{(2)}(z)e=\left[H^{(1)}(z)+H^{(2)}(z)\right]e\\
		&H(z)[H^{(1)}(z)e]=\left[H(z)H^{(1)}(z)\right]e\\
		&\alpha[H^{(1)}(z)e]=\left[\alpha H^{(1)}(z)\right]e.
	\end{align*}
\end{propos}
\begin{proof}
	Let
	\begin{align}
		H(z)=\sum_{k=-\infty}^{+\infty}h(k)z^{-k}.
	\end{align}

	\begin{itemize}
		\item	Sum.\\
		We have
		\begin{align*}
			&x_1(t)+x_2(t)=\\
			&\left[\sum_{k=-\infty}^{\infty} h^{(1)}_k e(t-k)\right]+
			\left[\sum_{k=-\infty}^{\infty} h^{(2)}_k e(t-k)\right]=\\
			&\sum_{k=-\infty}^{\infty} [h^{(1)}_k+h^{(2)}_k] e(t-k)=
			([H^{(1)}(z)+H^{(2)}(z)]e)(t).
		\end{align*}
		Since $[H^{(1)}(z)+H^{(2)}(z)]$ has no poles on the set $\{z\in\C|~|z|=1\}$, $x_1+x_2\in \mathcal{F}e$.
		\item	Multiplication by $H(z)\in \mathcal{F}$.\\
		Since $x_1 \in \mathcal{F}e$, it is a $1$-dimensional rationally related vector. Then, it makes sense to compute the random process $H(z)x_1$ for $H(z)\in \mathcal{F}=\mathcal{F}^{1 \times 1}$
		\begin{align}
			& (H(z)x_1)(t):=\sum_{k=-\infty}^{\infty} h(k) x_1(t-k)=\\
			   & =\sum_{k=-\infty}^{\infty} h(k) \sum_{l=-\infty}^{\infty} h^{(1)}_{l} e(t-k-l)=\\
			&=\sum_{k=-\infty}^{\infty} h(k) \sum_{l=-\infty}^{\infty} h^{(1)}_{l-k}e(t-l)=\\
			&=\sum_{l=-\infty}^{\infty}\left[\sum_{k=-\infty}^{\infty} h(k)  h^{(1)}_{l-k}\right]e(t-l)=\\
			&=\left(\left[H(z)H^{(1)}(z)\right]e\right)(t),
		\end{align}
		where the last equality comes from the properties of the convolution.
		Since $H(z)H^{(1)}(z)$ has no poles in the set $\{z\in\C|~|z|=1\}$, $H(z)x_1\in \mathcal{F}e$
		\item Multiplication by scalar $\alpha\in\Re$.\\
		It is a special case of the previous property.$\hfill\square$
	\end{itemize}
\end{proof}
\begin{define}
	We define a scalar binary operation $<\cdot,\cdot>$ on $\mathcal{F}e$ in the following way
	\begin{align*}
		<x_1,x_2>:= R_{x_1 x_2}(0).
	\end{align*}
\end{define}	
\begin{propos}
	The set $\mathcal{F}e$, along with the operation $<\cdot,\cdot>$ is a pre-Hilbert space (with the technical assumption that $x_1$ and $x_2$ are the same processes if $x_1 \sim x_2$).
\end{propos}
\begin{proof}
	The proof is left to the reader, making use of the introduced notations and properties. $\hfill\square$
\end{proof}

\end{document}